\documentclass[10pt,a4paper,reqno]{amsart}
\usepackage{amsthm}
\usepackage{amsmath}
\usepackage{amssymb}
\usepackage{graphicx,color}

\usepackage[font=small]{caption}
\usepackage{subcaption}

\usepackage{multirow}
\usepackage[shortlabels]{enumitem}

\makeatletter
\theoremstyle{plain}
\newtheorem{thm}{Theorem}
  \theoremstyle{definition}
  \newtheorem*{thm*}{Theorem}
  \newtheorem{defn}[thm]{Definition}
  \theoremstyle{remark}
  \newtheorem{rem}[thm]{Remark}
  \theoremstyle{plain}
  \newtheorem{prop}[thm]{Proposition}
  \theoremstyle{plain}
  \newtheorem{lem}[thm]{Lemma}
  \theoremstyle{plain}
  \newtheorem{cor}[thm]{Corollary}
 \theoremstyle{definition}
  \newtheorem{example}[thm]{Example}
  \theoremstyle{remark}
  \newtheorem*{rem*}{Remark}

  \theoremstyle{definition}

\usepackage{amsfonts}
\usepackage{mathrsfs}

\newtheorem*{question*}{\it{QUESTION}}

\theoremstyle{plain}
\newtheorem{conjecture}{Conjecture}

\newcommand{\N}{\mathbb{N}}
\newcommand{\R}{{\mathbb{R}}}
\newcommand{\C}{{\mathbb{C}}}

\newcommand{\Z}{{\mathbb{Z}}}
\newcommand{\dd}{{\rm d}}
\newcommand{\ii}{{\rm i}}

\newcommand{\p}{{\rm p}}
\newcommand{\ac}{{\rm ac}}

\newcommand{\spn}{\mathop\mathrm{span}\nolimits} 
\newcommand{\Dom}{\mathop\mathrm{Dom}\nolimits}

\newcommand{\Ran}{\mathop\mathrm{Ran}\nolimits}
\newcommand{\Ker}{\mathop\mathrm{Ker}\nolimits}

\renewcommand{\Re}{\mathop\mathrm{Re}\nolimits}
\renewcommand{\Im}{\mathop\mathrm{Im}\nolimits}
\newcommand{\supp}{\mathop\mathrm{supp}\nolimits}
\newcommand{\Tr}{\mathop\mathrm{Tr}\nolimits}

\def\pFq#1#2#3#4#5{ 
  {}_{#1}F_{#2}\biggl(\genfrac..{0pt}{}{#3}{#4}\biggl|\,#5\biggr)
}

\def\regpFq#1#2#3#4#5{ 
  {}_{#1}\tilde{F}_{#2}\biggl(\genfrac..{0pt}{}{#3}{#4}\biggl|\,#5\biggr)
}

\def\pPhiq#1#2#3#4#5{ 
  {}_{#1}\phi_{#2}\biggl(\genfrac..{0pt}{}{#3}{#4}\biggl|\,q;\,#5\biggr)
}

\makeatother

\begin{document}

\title[]{The Hilbert $L$-matrix}

\author{Franti\v sek \v Stampach}
\address[Franti{\v s}ek {\v S}tampach]{
	Department of Mathematics, Faculty of Nuclear Sciences and Physical Engineering, Czech Technical University in Prague, Trojanova~13, 12000 Praha~2, Czech Republic
	}	
\email{stampfra@fjfi.cvut.cz}



\date{\today}

\begin{abstract}
We analyze spectral properties of the Hilbert $L$-matrix 
\[
 \left(\frac{1}{\max(m,n)+\nu}\right)_{m,n=0}^{\infty}
\]
regarded as an operator $L_{\nu}$ acting on $\ell^{2}(\N_{0})$, for $\nu\in\R$, $\nu\neq0,-1,-2,\dots$. The approach is based on a spectral analysis of the inverse of $L_{\nu}$, which is an unbounded Jacobi operator whose spectral properties are deducible in terms of the unit argument ${}_{3}F_{2}$-hypergeometric functions. In particular, we give answers to two open problems concerning the operator norm of $L_{\nu}$ published by L.~Bouthat and J.~Mashreghi in [\emph{Oper. Matrices} 15, No.~1 (2021), 47--58]. In addition, several general aspects concerning the definition of an $L$-operator, its positivity, and Fredholm determinants are also discussed.
\end{abstract}

\maketitle

\section{Introduction}

In~\cite{bou-mas_oam21a}, L.~Bouthat and J.~Mashreghi use the Schur test to deduce sufficient conditions under which the structured matrix 
\[
 \mathcal{L}=\begin{pmatrix}
 				a_{0} & a_{1} & a_{2} & \dots\\
 				a_{1} & a_{1} & a_{2} & \dots\\
  				a_{2} & a_{2} & a_{2} & \dots\\
  				\vdots & \vdots & \vdots & \ddots
			 \end{pmatrix}
\]
determines a bounded operator on~$\ell^{2}(\N_{0})$ and derived an upper bound on its operator norm. We adopt the terminology from~\cite{bou-mas_oam21a} and refer to these matrices as the $L$-matrices.

\begin{defn}\label{def:L-matrix}
A semi-infinite matrix $\mathcal{L}$ is called \emph{$L$-matrix} if its matrix elements are of the form
\[
\mathcal{L}_{m,n}=a_{\max(m,n)}, \quad m,n\in\N_{0},
\]
where $a=\{a_{n}\}_{n=0}^{\infty}$ is a complex sequence. We call $a$ \emph{the parameter sequence} of $\mathcal{L}$.
\end{defn}

As an interesting example, the authors of~\cite{bou-mas_oam21a} studied the norm of a bounded operator $L_{\nu}$ determined by the $L$-matrix $\mathcal{L}_{\nu}$ whose parameter sequence reads
\[
 a_{n}(\nu)=\frac{1}{n+\nu},
\]
for $\nu>0$. In an obvious analogy to the famous (generalized) Hilbert matrix, which is the Hankel matrix with the $(m,n)$-th entry equal to $a_{m+n}(\nu)$, matrix $\mathcal{L}_{\nu}$ is called the \emph{Hilbert $L$-matrix}. It is proved in~\cite{bou-mas_oam21a} that 
\[
 \|L_{\nu}\|=4,
\]
if $\nu\geq 1/2$. However, for $\nu\in(0,1/2)$, the value of $\|L_{\nu}\|$ remained indeterminate. Since $\|L_{\nu}\|$ has to dominate the first diagonal element of $\mathcal{L}_{\nu}$, i.e., $\|L_{\nu}\|\geq a_{0}(\nu)=1/\nu$, it is clear that, for $\nu\in(0,1/4)$, $\|L_{\nu}\|>4$. This let the authors of~\cite{bou-mas_oam21a} to formulate two natural questions on the determination of numbers $\nu_{0}:=\inf\{\nu>0 \mid \|L_{\nu}\|=4\}$ and $\|L_{\nu}\|$, for $\nu<1/2$. In recent paper~\cite{bou-mas_laa22}, the authors deduced explicit upper and lower bounds on $\nu_{0}$.

The goal of this paper is a detailed spectral analysis of $L_{\nu}$ for all admissible real values of the parameter $\nu$, i.e., $\nu\neq0,-1,-2,\dots$. In particular, we provide answers to both above questions. An important role is played by the unit argument hypergeometric ${}_{3}F_{2}$-functions; see definition~\eqref{eq:def_3F2} below. We show that $\nu_{0}$ is the unique positive zero of the function
\[
 \nu\mapsto\pFq{3}{2}{-1/2,1/2,3/2}{1,\nu+1/2}{1},
\]
which, if evaluated numerically, is approximately $\nu_{0}\approx0.349086$. Further, we prove that, for $\nu\geq \nu_{0}$, the spectrum of $L_{\nu}$ is purely absolutely continuous and fills the interval $[0,4]$. However, if $0<\nu<\nu_{0}$, a unique and simple eigenvalue greater than $4$ appears in the spectrum of $L_{\nu}$. Clearly, this eigenvalue coincides with $\|L_{\nu}\|$. More concretely, we prove that
\[
 \|L_{\nu}\|=\frac{4}{1-4x_{0}^{2}(\nu)},
\]
where $x_{0}(\nu)$ is the unique zero of the function
\[
x\mapsto\pFq{3}{2}{x-1/2,x+1/2,x+3/2}{2x+1,x+\nu+1/2}{1}
\]
located in $(0,1/2)$. In addition, we show that $\|L_{\nu}\|$ is strictly decreasing in $(0,\nu_{0})$, deduce an asymptotic expansion of $\|L_{\nu}\|$ for $\nu$ small, and obtain bounds on $\|L_{\nu}\|$.

If $\nu<0$, spectral properties of $L_{\nu}$ are more delicate. It turns out that 
still $\sigma_{\ac}(L_{\nu})=[0,4]$ and $\sigma_{\p}(L_{\nu})$ can be identified with a set of roots of a~transcendental equation. Moreover, $\sigma_{\p}(L_{\nu})$ is shown to be always finite but, in contrast to the case $\nu>0$, $\sigma_{\p}(L_{\nu})$ contains a negative eigenvalue, if $\nu<0$. By a numerical evidence, it seems that $\sigma_{\text{p}}(L_{\nu})$ is actually at most a two-point set with a negative eigenvalue and possibly another eigenvalue greater than~4. This can be compared with the spectrum of the classical Hilbert matrix operator $H_{\nu}$, whose $(m,n)$-th matrix entry equals $a_{m+n}(\nu)$. Recall that $\sigma_{\text{ac}}(H_{\nu})=[0,\pi]$ for all $\nu\in\R\setminus(-\N_{0})$, while $\sigma_{\text{p}}(H_{\nu})$ is empty, if $\nu\geq1/2$, contains a unique eigenvalue greater than $\pi$, if $\nu\in(0,1/2)$, contains a negative eigenvalue, if $\nu<0$, and possibly another eigenvalue greater than $\pi$ (which is always the case, if $\nu<-1/2$). On the other hand, the spectrum of $L_{\nu}$ is always simple, which is not the case for $H_{\nu}$; see \cite[Thm.~5]{ros_pams58} or \cite[Thm.~8]{kal-sto_lma16} for more details. 

As far as the literature is concerned, let us mention that $L$-matrices were encountered in connection with the Hadamard multipliers in function spaces in~\cite{mas_09,mas-ran_amp19}. Some properties of $L$-matrices with a lacunary parameter sequence were studied in~\cite{bou-mas_oam21b}. The Hilbert $L$-matrix $\mathcal{L}_{1}$ also appears in Choi's paper~\cite{cho_amm83} as the \emph{loyal companion} of the Hilbert matrix. On the other hand, except the aforementioned works, it seems that general algebraic or analytic properties of $L$-matrices has not been systematically studied yet, also interesting concrete examples of such matrices with a solvable spectral problem are missing to author's best knowledge. Let us remark that (weighted) $L$-matrices can be find as resolvent operators of certain Jacobi operators that appear in problems of mathematical physics; see, for instance,  Example~\ref{ex:simple_L_matrix} below. This is related to an important fact that the formal inverse of an invertible $L$-matrix is tridiagonal; see Lemma~\ref{lem:tridiagonal_inverse}.

Although the main goal of this paper is the spectral analysis of operator $L_{\nu}$, we briefly discuss some general properties of $L$-matrices in Section~\ref{sec:L-matrices} and illustrate them on concrete examples.  We take the advantage of a close connection of $L$-matrices to Jacobi operators and orthogonal polynomials whose properties were extensively studied in the past and which theory is deeply developed. In Subsection~\ref{subsec:L-oper}, we show that to any given regular $L$-matrix (i.e., $a_{n}\neq a_{n+1}$ for all $n\in\N_{0}$), one can prescribe a unique densely defined $L$-operator acting on $\ell^{2}(\N_{0})$. The main advantage of this definition is that it determines the $L$-operator also in cases when the standard construction is not applicable, i.e., when $a\notin\ell^{2}(\N_{0})$. 
Next, in Subsection~\ref{subsec:positivity} and~\ref{subsec:determinant}, we characterize positive $L$-operators, give a sufficient condition for an $L$-operator $L$ to be of trace class, provide a~limit formula for the Fredholm determinant $\det(1-zL)$, and illustrate these results on an example of the $L$-operator with exponential parameter sequence. These results, although not needed for the spectral analysis of the Hilbert $L$-operator $L_{\nu}$, are of independent interest.

Main results focused on $L_{\nu}$ are derived in
Section~\ref{sec:Hilbert_L-matrix}. The spectral analysis of $L_{\nu}$ is worked out via the spectral analysis of the inverse $J_{\nu}:=L_{\nu}^{-1}$, which turns of to be an unbounded Jacobi operator of a special structure. 
In Subsection~\ref{subsec:nu=1}, it is shown that, in the particular case $\nu=1$, Jacobi operator $J_{1}$ corresponds to a subfamily of Continuous dual Hahn orthogonal polynomials, which allows an almost immediate spectral analysis of $L_{1}$. The spectral analysis of $J_{\nu}$, for general $\nu\in\R\setminus(-\N_{0})$, is started in Subsection~\ref{subsec:eigvectors_of_J_nu}, where essential formulas for generalized eigenvectors in terms of hypergeometric ${}_{3}F_{2}$-functions are established and their properties are investigated. The analysis continues in Subsection~\ref{subsec:spectrum_of_J_nu}, where the spectrum of $J_{\nu}$ and its parts are determined and, more generaly, formulas for the Weyl $m$-function and the spectral measure of $J_{\nu}$ are deduced. More detailed analysis of $\sigma_{\p}(J_{\nu})$, for $\nu>0$, which allows to answer the open problems from~\cite{bou-mas_oam21a}, is done in Subsection~\ref{subsec:point_spectrum_J_nu_nu>0}. Corollaries on properties of orthogonal polynomials associated to~$J_{\nu}$ are summarized in Subsection~\ref{subsec:ortohogonal_polynomials}. Finally, in Subsection~\ref{subsec:spectrum_L_nu}, main results on spectral properties of $L_{\nu}$ are formulated as Theorems~\ref{thm:spectrum_L_nu} and~\ref{thm:spectrum_L_nu_nu>0} and a conjecture is provided concerning the fine structure of $\sigma_{\p}(L_{\nu})$ for $\nu<0$.

Several illustrative and numerical plots are given in Section~\ref{sec:numerics}. The paper is concluded by Appendix divided into two parts. First, for reader's convenience, we recall necessary steps of the method of successive approximation~\cite{li-wong_jcam92} that are used in the paper. Second part contains computational details on a derivation of higher order terms in an asymptotic expansion of the root $x_{0}(\nu)$ for $\nu$ small.

\textbf{Notation:} As one can observe above, we distinguish in notation between semi-infinite matrix and an operator acting on $\ell^{2}(\N_{0})$ by using the capital calligraphic letters for matrices and standard capitals for operators. This is, of course, not needed when dealing with bounded operators only, however, necessary in general. Moreover, we use the following standard notation throughout the paper: $\N_{0}$ is the set of non-negative integers; $\N$ is the set of positive integers; $\{e_{n}\mid n\in\N_{0}\}$ is the standard basis of $\ell^{2}(\N_{0})$; $\Dom A$, $\Ker A$, and $\Ran A$ is the domain, kernel, and range of an operator $A$, respectively; $\sigma(A)$, $\sigma_{\p}(A)$, and $\sigma_{\ac}(A)$ is the spectrum, point spectrum, and absolutely continuous spectrum of $A$, respectively. 

Other notation, mainly related to special functions, is defined by its first occurrence.

\section{$L$-matrices}\label{sec:L-matrices}

\subsection{An operator associated to a regular $L$-matrix}\label{subsec:L-oper}

We first recall a general method of constructing densely defined operators on $\ell^{2}(\N_{0})$ with a given matrix representation, see for example~\cite[Sec.~2.1]{bec_jcam00}. A semi-infinite matrix $\mathcal{A}$ determines a closed and densely defined operator on $\ell^{2}(\N_{0})$ when its rows and columns can be identified with elements of~$\ell^{2}(\N_{0})$. In this case, we can define the so-called minimal operator $A_{\min}$ and the maximal operator $A_{\max}$ associated with the matrix $\mathcal{A}$.

The minimal operator $A_{\min}$ associated with matrix $\mathcal{A}$ is defined as the operator closure of the auxiliary operator $\dot{A}$ acting on vectors $x$ from the space of finitely supported sequences $\Dom\dot{A}:=\spn\{e_{n}\mid n\in\N_{0}\}$ as the matrix multiplication $\dot{A}x:=\mathcal{A}\cdot x$, where $x$ is regarded as an infinite column vector. The fact that $\dot{A}$ is always closable can be easily checked.
On the other hand, the maximal operator $A_{\max}$ acts again as the matrix multiplication $A_{\max}x:=\mathcal{A}\cdot x$ but on vectors $x$ from a maximal domain
\[
\Dom A_{\max}:=\left\{x\in\ell^{2}(\N_{0}) \mid \mathcal{A}\cdot x\in\ell^{2}(\N_{0})\right\}.
\]
$A_{\max}$ is always closed.

One has $A_{\min}\subset A_{\max}$. If $B$ is any operator which acts by the matrix multiplication $Bx=\mathcal{A}\cdot x$ and $\spn\{e_{n}\mid n\in\N_{0}\}\subset\Dom B$, then $A_{\min}\subset B\subset A_{\max}$. In this sense, matrix~$\mathcal{A}$ determines a unique operator $A$ acting on $\ell^{2}(\N_{0})$ with $\langle e_{m},Ae_{n}\rangle=\mathcal{A}_{m,n}$, for all $m,n\in\N_{0}$, if and only if $A_{\min}=A_{\max}$. Such matrix~$\mathcal{A}$ is called \emph{proper}.

Clearly, if $\mathcal{A}$ is an $L$-matrix, then its rows and columns are square summable if and only if its parameter sequence $a\in\ell^{2}(\N_{0})$.

\begin{defn}\label{def:sing_reg_L-matrix}
An $L$-matrix $\mathcal{L}$ is called \emph{regular}, if its parameter sequence fulfills $a_{n}\neq a_{n+1}$ for all $n\in\N_{0}$, and \emph{singular}, otherwise.
\end{defn}

The following claim justifies the terminology used in Definition~\ref{def:sing_reg_L-matrix}. 

\begin{lem}
 Let $\mathcal{L}$ be an $L$-matrix with $a\in\ell^{2}(\N_{0})$. Then
 \[
  \mathcal{L} \mbox{ is regular } \quad\Leftrightarrow\quad \Ker L=\{0\},
 \]
 for any operator $L$ such that $L_{\min}\subset L\subset L_{\max}$.
\end{lem}

\begin{proof}
 Fix $L$ a linear operator with $\spn\{e_{n}\mid n\in\N_{0}\}\subset \Dom L$ and $\mathcal{L}_{m,n}=\langle e_{m},Le_{n}\rangle$ for all $m,m\in\N_{0}$.

 Suppose $\mathcal{L}$ is singular. Then there exists $n\in\N_{0}$ such that $a_{n}=a_{n+1}$ which means that the $n$th and $(n+1)$th column of $\mathcal{L}$ coincide. It follows that $e_{n}-e_{n+1}\in\Ker L$.
 
 Suppose $\mathcal{L}$ is regular and $x\in\Ker L$. Then the equation $Lx=0$ means that
 \[
 a_{n}\sum_{k=0}^{n}x_{k}+\sum_{k=n+1}^{\infty}a_{k}x_{k}=0, \quad \forall n\in\N_{0}.
 \]
 Shifting the index $n$ by one and subtracting the resulting equation from the above one, we obtain
 \[
 (a_{n+1}-a_{n})\sum_{k=0}^{n}x_{k}=0, \quad \forall n\in\N_{0}.
 \]
 Since $a_{n}\neq a_{n+1}$, for all $n\in\N_{0}$, by assumption, the above equations imply that $x=0$.
\end{proof}

A particularly useful observation is that the formal inverse of a regular $L$-matrix is a~tridiagonal matrix with a special structure. Recall that the matrix multiplication $\mathcal{A}\cdot\mathcal{B}$ of infinite matrices $\mathcal{A}$ and $\mathcal{B}$ is well defined, if at least one of the matrices $\mathcal{A}$ or $\mathcal{B}$ is banded. Provided that $a_{n}\neq a_{n+1}$ for all $n\in\N_{0}$, we put
\begin{equation}
 b_{n}:=\frac{1}{a_{n}-a_{n+1}}, \quad n\in\N_{0},
\label{eq:def_b_n}
\end{equation}
and define the Jacobi matrix
\begin{equation}
 \mathcal{J}:=\begin{pmatrix}
 				b_{0} & -b_{0} \\
 				-b_{0} & b_{0}+b_{1} & -b_{1}\\
  				& -b_{1} & b_{1}+b_{2} & -b_{2}\\
  				& & -b_{2} & b_{2}+b_{3} & -b_{3}\\
  				& & & \ddots & \ddots & \ddots
			 \end{pmatrix}.
\label{eq:def_J}
\end{equation}
Matrix $\mathcal{J}$ turns out to be a formal inverse to regular $L$-matrix $\mathcal{L}$.

\begin{lem}\label{lem:tridiagonal_inverse}
 Let $\mathcal{L}$ be a regular $L$-matrix and $\mathcal{J}$ the Jacobi matrix defined by~\eqref{eq:def_b_n} and~\eqref{eq:def_J}, then we have:
 \begin{enumerate}[{\upshape i)}]
 \item $\mathcal{L}\cdot\mathcal{J}=\mathcal{J}\cdot\mathcal{L}=\mathcal{I}$,  where $\mathcal{I}$ is the identity matrix.
 \item If $\mathcal{J}\cdot x=0$, then $x_{n}=x_{0}$ for all $n\in\N_{0}$.
 \item $\mathcal{J}$ is proper.
 \end{enumerate}
\end{lem}

\begin{proof}
The proof of (i) is a matter of straightforward matrix multiplication. The verification of (ii) is also elementary. Thus, it suffices to prove claim~(iii). 

Let us denote by $p=\{p_{n}\}_{n=0}^{\infty}$ and $q=\{q_{n}\}_{n=0}^{\infty}$ the two solutions of the second order difference equation
\[
 -b_{n-1}\phi_{n-1}+(b_{n-1}+b_{n})\phi_{n}-b_{n}\phi_{n-1}=0, \quad n\in\N,
\]
determined by the initial conditions $p_{0}=0$, $p_{1}=1$ and $q_{0}=1$, $q_{1}=0$. The matrix $\mathcal{J}$ is proper, if and only if $p\notin\ell^{2}(\N_{0})$ or $q\notin\ell^{2}(\N_{0})$; see~\cite{bec-smi_ms04}. Using~\eqref{eq:def_b_n}, a straightforward computation shows that
\[
 p_{n}=\frac{a_{0}-a_{n}}{a_{0}-a_{1}} \quad\mbox{ and }\quad  q_{n}=\frac{a_{n}-a_{1}}{a_{0}-a_{1}},
\]
for $n\in\N_{0}$. Since $p_{n}+q_{n}=1$ for all $n\in\N$, $p$ and $q$ cannot be both in $\ell^{2}(\N_{0})$.
\end{proof}

Observations from Lemma~\ref{lem:tridiagonal_inverse} can be used to define a unique closed operator associated to a regular $L$-matrix even when $a\notin\ell^{2}(\N_{0})$, i.e, when the standard construction is not applicable. Indeed, since $\mathcal{J}$ is proper by claim~(iii) of Lemma~\ref{lem:tridiagonal_inverse}, it determines the unique operator $J$ acting on $\ell^{2}(\N_{0})$. Moreover, $J$ is invertible since $\Ker J=\{0\}$ by claim~(ii) of Lemma~\ref{lem:tridiagonal_inverse}. This justifies the following definition.

\begin{defn}\label{def:L-operator}
 To a given regular $L$-matrix $\mathcal{L}$, we associate the $L$-operator $L:=J^{-1}$, where $J$ is the Jacobi operator given by~\eqref{eq:def_J}. 
\end{defn}

\begin{rem}
 We also call the $L$-operator from Definition~\ref{def:L-operator} \emph{regular}, similarly as the $L$-matrix, to emphasize the (implicit) assumption $a_{n}\neq a_{n+1}$, for all $n\in\N_{0}$.
\end{rem}

\begin{prop}
 A regular $L$-operator is closed and densely defined.
\end{prop}

\begin{proof}
As an inverse of a closed operator $J$, the regular $L$-operator $L$ is closed.
Further, $(\Ran J)^{\perp}=\Ker J^{*}$ and, for $x\in\Ker J^{*}$, one has $\mathcal{J}^{*}\cdot x=0$, which implies $x=0$ by claim~(ii) of Lemma~\ref{lem:tridiagonal_inverse}, where $\mathcal{J}^{*}$ denotes the Jacobi matrix~\eqref{eq:def_J} with complex conjugated entries. It implies that $\Dom L$ is dense in $\ell^{2}(\N_{0})$.
\end{proof}

Of course, for an $L$-operator $L$  associated to a regular $L$-matrix $\mathcal{L}$, the matrix $\mathcal{L}$ is not the matrix representation of $L$ in the usual sense, which means
\[
 \mathcal{L}_{m,n}=\langle e_{m},Le_{n}\rangle, \quad \forall m,n\in\N_{0},
\]
because the vectors $e_{n}$ need not belong to $\Dom L$. This is guaranteed when we additionally have the parameter sequence $a\in\ell^{2}(\N_{0})$.

Even if $a\in\ell^{2}(\N_{0})$, the $L$-operator need not be bounded. For example, if $a_{n}=1/(n+1)^{\alpha}$, for $1/2<\alpha<1$, then $L$ is unbounded~\cite[Ex.~2]{bou-mas_oam21a}. A sufficient condition for boundedness of $L$ is~$a_{n}=O(1/n)$, as $n\to\infty$; see~\cite{bou-mas_oam21a}. In view of Definition~\ref{def:L-operator}, $L$ is bounded if and only if $0\notin\sigma(J)$. Recall that $0\notin\sigma_{\p}(J)$ as we know by claim~(ii) of Lemma~\ref{lem:tridiagonal_inverse}.

Note also that the mapping $\mathcal{L}\mapsto L$ between regular $L$-matrices and $L$-operators established by Definition~\ref{def:L-operator} is not injective. Indeed, it suffices to note that the matrix elements of~$\mathcal{J}$ depend only on the \emph{difference} between two consecutive elements of the parameter sequence $a$. Hence, an $L$-matrix with a parameter sequence $\{a_{n}\}_{n=0}^{\infty}$ gives rise to the same $L$-operator as the $L$-matrix with the parameter sequence $\{a_{n}+c\}_{n=0}^{\infty}$, where $c\in\C$. This ambiguity is removed when the admissible parameter sequences are restricted to the square summable.

\begin{example}\label{ex:simple_L_matrix}
 Consider the parameter sequence $a_{n}=n$, $n\in\N_{0}$. Then the $L$-operator $L$ associated to the $L$-matrix
 \[
  \mathcal{L}=\begin{pmatrix}
  0 & 1 & 2 & 3 & \dots\\
  1 & 1 & 2 & 3 & \dots\\
  2 & 2 & 2 & 3 & \dots\\
  3 & 3 & 3 & 3 & \dots\\
  \vdots & \vdots & \vdots & \vdots & \ddots
  \end{pmatrix}
 \]
 is the inverse of the Jacobi operator $J$ given by the tridiagonal matrix
 \[
  \mathcal{J}=\begin{pmatrix}
 				-1& 1 \\
 				1 & -2 & 1\\
  				& 1 & -2 & 1\\
  				& & 1 & -2 & 1\\
  				& & & \ddots & \ddots & \ddots
			 \end{pmatrix}.
 \]
 Operator $J$ is a discrete analogue to the Neumann Laplacian on the half-line and its spectral analysis is not very difficult~\cite{lap-kre-sta_prep21}. It can be shown that $J$ has simple spectrum 
\[ 
 \sigma(J)=\sigma_{\ac}(J)=[-4,0] 
\] 
and
\[
 \psi_{n}(\lambda):=U_{n}\left(1+\frac{\lambda}{2}\right)-U_{n-1}\left(1+\frac{\lambda}{2}\right), \quad n\in\N_{0},
\]
 is the $n$th entry of the generalized eigenvector to $\lambda\in[-4,0]$, where $U_{n}$ are Chebyshev polynomials of the second kind and $U_{-1}:=0$. In particular, it follows that the spectrum of~$L$ is simple and $\sigma(L)=\sigma_{\ac}(L)=(-\infty,-1/4]$.
\end{example}

\subsection{Factorization and positivity of regular $L$-matrices}\label{subsec:positivity}

First, it is useful to observe that, if $\mathcal{C}$ is a bi-diagonal semi-infinite matrix of the form
\begin{equation}\label{eq:matrix_C}
 \mathcal{C}=\begin{pmatrix}
  c_{0} & -c_{0} \\
  &   c_{1} & -c_{1} \\
  & &   c_{2} & -c_{2} \\
  & & & \ddots & \ddots
 \end{pmatrix},
\end{equation}
then
\[
 \mathcal{C}^{T}\cdot\mathcal{C}=\begin{pmatrix}
  c_{0}^{2} & -c_{0}^{2} \\
  -c_{0}^{2} & c_{0}^{2}+c_{1}^{2} & -c_{1}^{2} \\
  &   -c_{1}^{2} & c_{1}^{2}+c_{2}^{2} & -c_{2}^{2} \\
  & & \ddots & \ddots & \ddots 
 \end{pmatrix}.
\]
Comparing this observation with the structure of the Jacobi matrix $\mathcal{J}$ from~\eqref{eq:def_J}, it is clear that $\mathcal{J}$ decomposes as
\begin{equation}\label{eq:factor_J}
\mathcal{J}=\mathcal{B}^{T}\cdot\mathcal{B},
\end{equation}
where
\[
\mathcal{B}=\begin{pmatrix}
  \sqrt{b_{0}} & -\sqrt{b_{0}} \\
  &   \sqrt{b_{1}} & -\sqrt{b_{1}} \\
  & &   \sqrt{b_{2}} & -\sqrt{b_{2}} \\
  & & & \ddots & \ddots
 \end{pmatrix},
\]
provided that $b_{n}>0$, for all $n\in\N_{0}$.

\begin{prop}\label{prop:positivity_L_oper}
A real regular $L$-operator is positive semi-definite if and only if $a_{n}>a_{n+1}$, for all $n\in\N_{0}$.
\end{prop}

\begin{proof}
Clearly, an $L$-operator $L=J^{-1}$ is positive semi-definite if and only if the Jacobi operator $J$ is positive semi-definite. Also $a_{n}>a_{n+1}$, for all $n\in\N_{0}$, if and only if $b_{n}>0$, for all $n\in\N_{0}$, which follows readily from~\eqref{eq:def_b_n}. Consequently, we need to show that $J$ given by~\eqref{eq:def_J} is positive semi-definite if and only if $b_{n}>0$, for all $n\in\N_{0}$.

Suppose that $J$ is positive semi-definite, i.e., $\langle x,Jx\rangle\geq0$, for all $x\in\Dom J$. It particularly follows that
\[
 \det J_{n}\geq0, \quad \forall n\in\N,
\]
where $J_{n}\in\C^{n,n}$ denotes the $n\times n$ section of the matrix~$\mathcal{J}$, i.e., $J_{n}=P_{n}JP_{n}\upharpoonleft\Ran P_{n}$, where $P_{n}$ is the orthogonal projection onto $\spn\{e_{0},\dots,e_{n-1}\}$. It is a matter of simple linear algebra to verify that
\[
\det J_{n}=\prod_{k=0}^{n-1}b_{k}, \quad \forall n\in\N.
\]
Taking into account that $b_{n}\neq0$ by definition, the positivity of $b_{n}$, for all $n\in\N_{0}$, follows by induction.

Suppose, on the contrary, that $b_{n}>0$, $\forall n\in\N_{0}$. Decomposition~\eqref{eq:factor_J} implies that, for any $x\in\spn\{e_{n} \mid n\in\N_{0}\}$, one has
\[
\langle x,Jx\rangle=\|\mathcal{B}x\|^{2}\geq0.
\]
Since $\mathcal{J}$ is proper by claim~(iii) of Lemma~\ref{lem:tridiagonal_inverse}, $J=J_{\min}$. Therefore to any $x\in\Dom J$, there exists a sequence of vectors $x_{n}\in\spn\{e_{n} \mid n\in\N_{0}\}$ such that $x_{n}\to x$ and $Jx_{n}\to Jx$ in $\ell^{2}(\N_{0})$, as $n\to\infty$. Then $\langle x_{n},Jx_{n}\rangle\to\langle x,Jx\rangle$, as $n\to\infty$, and the inequality
$\langle x,Jx\rangle\geq0$ extends to all $x\in\Dom J$.
\end{proof}

\begin{rem}
 Note that, if $c_{n}\neq0$ for all $n\in\N_{0}$, then matrix $\mathcal{C}$ from~\eqref{eq:matrix_C} is formally invertible with the inverse
 \[
  \mathcal{C}^{-1}=\begin{pmatrix}
  \frac{1}{c_{0}} & \frac{1}{c_{1}} & \frac{1}{c_{2}} & \ldots \\
   0              & \frac{1}{c_{1}} & \frac{1}{c_{2}} & \ldots \\
   0              &      0          & \frac{1}{c_{2}} & \ldots \\
	\vdots        &   \vdots        &      \vdots     & \ddots
 \end{pmatrix}.
 \]
 This observation yields the following factorization of an $L$-matrix
 \[
 \mathcal{L}=\mathcal{A}^{T}\cdot\mathcal{A},
 \]
 where
 \[
 \mathcal{A}=\begin{pmatrix}
 \sqrt{a_0-a_1} & 0 & 0 & \ldots\\
 \sqrt{a_1-a_2} &  \sqrt{a_1-a_2} & 0 & \ldots\\
 \sqrt{a_2-a_3} &  \sqrt{a_2-a_3} & \sqrt{a_2-a_3} & \ldots \\
  \vdots        &   \vdots        &      \vdots     & \ddots
 \end{pmatrix},
 \] 
 provided that $a_{n}>a_{n+1}$, for all $n\in\N_{0}$, and $a_n\to0$, as $n\to\infty$. Matrices of this structure appear in literature as terraced matrices~\cite{rha_blms89} or $C$-matrices~\cite{bou-mas_laa22}. A decomposition of the Hilbert $L$-matrix into a product of Ces\`{a}ro matrices is also mentioned in~\cite{cho_amm83}.
\end{rem}

\subsection{Schatten norms and the Fredholm determinant}\label{subsec:determinant}

First, we prove a simple-to-check sufficient condition for an $L$-operator to be in the trace class. For $p\geq1$, we denote by $\mathscr{J}_{p}$ the $p$-th Schatten class and $\|\cdot\|_{\mathscr{J}_{p}}$ the $p$-th Schatten norm. In particular, $\mathscr{J}_{1}$ is the ideal of trace class operators and  $\|A\|_{\mathscr{J}_{1}}=\Tr|A|$.

\begin{thm}\label{thm:l-matrix_trace_class}
If a complex sequence $a$ fulfills the condition
\begin{equation}
 \sum_{k=1}^{\infty}\sqrt{k}|a_{k}|<\infty,
\label{eq:asum_sum_trace_class}
\end{equation}
then the $L$-matrix given by the parameter sequence $a$ is a matrix representation of a trace class operator $L$ on $\ell^{2}(\N_{0})$. Moreover, for any $p\geq1$, we have the following estimate:
\[
 \|L\|_{\mathscr{J}_{p}}\leq\sum_{k=0}^{\infty}|a_{k}|\left((\mu_{+}(k))^{p/2}+(\mu_{-}(k))^{p/2}\right)^{1/p},
\]
where
\[
 \mu_{\pm}(k)=k+\frac{1}{2}\pm\sqrt{k+\frac{1}{4}}.
\]
\end{thm}

\begin{proof}
For $k\in\N_{0}$, we denote by $F_{k}$ the finite rank operator given by the $L$-matrix with parameter sequence $f_{k}=\{\delta_{k,n}\}_{n=0}^{\infty}$, where $\delta_{k,n}=1$, if $k=n$, and $\delta_{k,n}=0$, if $k\neq n$. One readily computes that, for $k\in\N$, the only nonzero eigenvalues of $F_{k}^{2}$ are $\mu_{\pm}(k)$ and these are simple. Consequently, for any $p\geq1$ and $k\in\N$, one has
\[
 \|F_{k}\|_{\mathscr{J}_{p}}=\left((\mu_{+}(k))^{p/2}+(\mu_{-}(k))^{p/2}\right)^{1/p}.
\]
Since $\|F_{0}\|_{\mathscr{J}_{p}}=1$, the last equality remains true also for $k=0$.
Then the assumption \eqref{eq:asum_sum_trace_class} guarantees that the series
\[
 \sum_{k=0}^{\infty}a_{k}F_{k}
\]
converges in the trace class norm and the limit is an operator $L$ whose matrix representation coincides with the $L$-matrix with parameter sequence $a$. The second claim from the statement follows from the estimate
\[
 \|L\|_{\mathscr{J}_{p}}\leq\sum_{k=0}^{\infty}|a_{k}|\|F_{k}\|_{\mathscr{J}_{p}}.
\]
\end{proof}

\begin{rem}
 Let us remark that condition~\eqref{eq:asum_sum_trace_class} is not necessary. Notice that, if $L$ is a regular $L$-operator with strictly decreasing parameter sequence~$a$, then $L\geq0$ by Proposition~\ref{prop:positivity_L_oper} and hence
 \[
  L\in\mathscr{J}_{1} \quad\Leftrightarrow\quad \|L\|_{\mathscr{J}_{1}}=\sum_{n=0}^{\infty}\left(a_{n}-a_{\infty}\right)<\infty,
 \]
 where $a_{\infty}:=\lim_{n\to\infty}a_{n}$. This means that, for example, $a_{n}=1/(n+1)^{\alpha}$ is a parameter sequence of a trace class $L$-operator, for any $\alpha>1$, however, condition~\eqref{eq:asum_sum_trace_class} is not fulfilled if $\alpha\in(1,3/2]$.
\end{rem}

Recall that for a trace class operator $L$, the Fredholm determinant $\det(1-zL)$, also referred to as the \emph{characteristic function} of $L$, is a well-defined entire function of~$z$ and its zeros coincides with the reciprocal values of non-zero eigenvalues of $L$ together with their multiplicities, see~\cite[Sec.~XIII.17]{reed-simon_IV}. Since there exist formulas for the characteristic function of certain Jacobi operators, see~\cite{sta_ieot17, sta-sto_laa13}, and one may also recall the Fredholm formula for the characteristic function of integral operators~\cite[Thm.~3.10]{simon_79}, it is likely that there is a formula for the characteristic function of a trace class $L$-operator expressed directly in terms of its parameter sequence. Such formula might be useful for spectral analysis of trace class $L$-operators, however, a more detailed investigation of these aspects should be a~subject of a~separate study. At this point, we provide only a limit formula for the characteristic function expressed in terms of the family of orthogonal polynomials corresponding to Jacobi matrix~\eqref{eq:def_J} and demonstrate its application on an interesting example below.

Let $\{p_{n}\}_{n=0}^{\infty}$ denotes the sequence of monic orthogonal polynomials determined by~\eqref{eq:def_J}, i.e., 
\begin{equation}
 p_{n}(z):=\det(z-J_{n}),
\label{eq:def_OGP}
\end{equation}
where $J_{n}:=P_{n}JP_{n}\upharpoonleft\Ran P_{n}$ and $P_{n}$ is the orthogonal projection onto $\spn\{e_{0},\dots,e_{n-1}\}$. Alternatively, polynomials $p_{n}$ are determined recursively by the recurrence
\[
 p_{n+1}(z)=(z-b_{n-1}-b_{n})p_{n}(z)-b_{n-1}^{2}p_{n-1}(z), \quad n\in\N,
\]
with the initial conditions $p_{0}(z)=1$ and $p_{1}(z)=z-b_{0}$. Recall~\eqref{eq:def_b_n}.

\begin{thm}\label{thm:char_func_L-oper}
 Suppose $L$ be a regular trace class $L$-operator with parameter sequence $a$. Then 
 \[
  \det(1-zL)=\lim_{n\to\infty}(-1)^{n}\left[\prod_{j=1}^{n-1}(a_{j-1}-a_{j})\right]\!\left(a_{n-1}p_{n}(z)+a_{n}b_{n-1}p_{n-1}(z)\right)\!,
 \]
 for all $z\in\C$.
\end{thm}

\begin{proof}
 First, recall that, if $L\in\mathscr{J}_{1}$, then 
 \[
 \det(1-zL)=\lim_{n\to\infty}\det(1-zP_{n}LP_{n})=\lim_{n\to\infty}\det(1-zL_{n}),
 \] 
 where $L_{n}$ denotes the $n\times n$ section of the $L$-matrix $\mathcal{L}$, analogously as for $J_{n}$. The identity operator $1$ acts either on $\ell^{2}(\N_{0})$ or $\C^{n}$ depending on the context without being explicitly distinguished in the notation. Thus, in the remaining part of the proof, it suffices to verify the identity
 \begin{equation}
 \det(1-zL_{n})=(-1)^{n}\left[\prod_{j=1}^{n-1}(a_{j-1}-a_{j})\right]\!\left(a_{n-1}p_{n}(z)+a_{n}b_{n-1}p_{n-1}(z)\right).
 \label{eq:det_L_n_OGP}
 \end{equation}
 
 Fix $n\in\N$ and suppose additionally that $a_{n-1}\neq0$. This means that $L_{n}$ is invertible since
 \begin{equation}
  \det L_{n}=a_{n-1}\prod_{j=1}^{n-1}(a_{j-1}-a_{j}).
 \label{eq:det_L_n}
 \end{equation}
 It is straightforward to check that
 \[
  L_{n}J_{n}=1+\frac{a_{n}b_{n-1}}{a_{n-1}}L_{n}e_{n}e_{n}^{T},
 \]
 where we denote $e_{n}:=(0,\dots,0,1)^{T}\in\C^{n}$ with some abuse of notation. It follows that
 \begin{equation}
  L_{n}^{-1}=J_{n}-\frac{a_{n}b_{n-1}}{a_{n-1}}e_{n}e_{n}^{T}.
  \label{eq:L_n_inverse}
 \end{equation}
 
 In view of~\eqref{eq:det_L_n} and~\eqref{eq:L_n_inverse}, we have
 \begin{align*}
  \det(1-zL_{n})&=(-1)^{n}z^{n}\det(L_{n})\det\!\left(1-\frac{1}{z}L_{n}^{-1}\right)\\
  				&=(-1)^{n}z^{n}a_{n-1}\left[\prod_{j=1}^{n-1}(a_{j-1}-a_{j})\right]\det\!\left(1-\frac{1}{z}J_{n}+\frac{a_{n}b_{n-1}}{z a_{n-1}}e_{n}e_{n}^{T}\right)
 \end{align*}
 Next, by applying the identity
 \begin{equation}
  \det(A+cd^{T})=\left(1+d^{T}A^{-1}c\right)\det A,
 \label{eq:determinant_A_plus_rank-one}
 \end{equation}
 which holds true for any $c,d\in\C^{n}$ and invertible matrix $A\in\C^{n,n}$, we see that $\det(1-zL_{n})$
 equals
 \[
  (-1)^{n}z^{n}a_{n-1}\left[\prod_{j=1}^{n-1}(a_{j-1}-a_{j})\right]\left[1+\frac{a_{n}b_{n-1}}{z a_{n-1}}e_{n}^{T}\left(1-\frac{1}{z}J_{n}\right)^{-1}e_{n}\right]\det\!\left(1-\frac{1}{z}J_{n}\right),
 \]
 provided that $z^{-1}\notin\sigma(J_{n})$. If we also recall the well-know relation between the inverse of an invertible matrix $A$ and the adjugate matrix of $A$, which particularly implies that
 \[
  e_{n}^{T}Ae_{n}=\frac{\det A[n]}{\det A},
 \]
 where $A[n]$ denotes the matrix $A$ whose last column and row are deleted, we obtain
 \begin{align*}
  \det(1-zL_{n})&=(-z)^{n}a_{n-1}\left[\prod_{j=1}^{n-1}(a_{j-1}-a_{j})\right]\left[\det\!\left(1-\frac{1}{z}J_{n}\right)+\frac{a_{n}b_{n-1}}{z a_{n-1}}\det\!\left(1-\frac{1}{z}J_{n-1}\right)\right]\\
  &=(-1)^{n}\left[\prod_{j=1}^{n-1}(a_{j-1}-a_{j})\right]\left[a_{n-1}\det\!\left(z-J_{n}\right)+a_{n}b_{n-1}\det\!\left(z-J_{n-1}\right)\right].
 \end{align*}
 If the last expression is written in terms of the orthogonal polynomials~\eqref{eq:def_OGP}, we arrive at~\eqref{eq:det_L_n_OGP}. Clearly, the additional assumptions $a_{n-1}\neq0$ and $z^{-1}\notin\sigma(J_{n})$ imposed during the derivation can be removed and formula~\eqref{eq:det_L_n_OGP} extends to all $z\in\C$ and arbitrary sequence $a$.
\end{proof}

We conclude this section by an interesting example of a trace class $L$-operator whose characteristic function can be computed with the aid of Theorem~\ref{thm:char_func_L-oper} in terms of well-known special functions.

\begin{example}\label{ex:L-matrix_exponential}
We investigate the $L$-operator $L$ with exponential matrix entries:
\[
 \mathcal{L}=\begin{pmatrix}
 				1 & q & q^{2} & \dots\\
 				q & q & q^{2} & \dots\\
  				q^{2} & q^{2} & q^{2} & \dots\\
  				\vdots & \vdots & \vdots & \ddots
			 \end{pmatrix},
\]
i.e., $a_{n}=q^{n}$, where $q\in(0,1)$. By Proposition~\ref{prop:positivity_L_oper} and Theorem~\ref{thm:l-matrix_trace_class}, $L$ is a positive semi-definite trace class operator. By~\eqref{eq:def_b_n}, we have
\[
 b_{n}=\frac{1}{a_{n+1}-a_{n}}=\frac{q^{-n}}{1-q}
\]
for $n\in\N_{0}$. Define also $b_{-1}:=q/(1-q)$. Our goal is to make use of known asymptotic properties of the $q$-Chebyshev polynomials studied by Ismail and Mulla in~\cite{ism-mul_siamjma87} to deduce an~asymptotic formula for the orthogonal polynomials $p_{n}(x)=\det(x-J_{n})$, for $n\to\infty$. Then an application of Theorem~\ref{thm:char_func_L-oper} provides us with the characteristic function of $L$. To this end, we introduce the rank-one perturbation matrix
\[
 \tilde{J}_{n}:=J_{n}+b_{-1}e_{0}e_{0}^{T},
\]
the corresponding orthogonal polynomials $\tilde{p}_{n}(x):=\det(x-\tilde{J}_{n})$, and also the associated orthogonal polynomials $\tilde{q}_{n-1}(x):=\det(x-\tilde{J}_{n}[1])$, where $\tilde{J}_{n}[1]$ is the matrix $\tilde{J}_{n}$ with deleted first row and column. Recalling once more identity~\eqref{eq:determinant_A_plus_rank-one}, we have
\begin{equation}
 p_{n}(x)=\left(1+b_{-1}e_{0}^{T}(x-\tilde{J}_{n}e_{0})^{-1}e_{0}\right)\det\left(x-\tilde{J}_{n}\right)=\tilde{p}_{n}(x)+b_{-1}\tilde{q}_{n-1}(x),
\label{eq:tow_p_n_ident}
\end{equation}
where we have used the identity
\[
 e_{0}^{T}(x-\tilde{J}_{n}e_{0})^{-1}e_{0}=\frac{\tilde{q}_{n-1}(x)}{\tilde{p}_{n}(x)}.
\]

By using the three-term recurrences for polynomials $\tilde{p}_{n}$ and $\tilde{q}_{n}$, one can relate them to the $q$-Chebyshev polynomials defined by the recurrence
\[
 \theta^{(a)}_{n+1}(x;q)=(2x-aq^{n})\theta^{(a)}_{n}(x;q)-\theta^{(a)}_{n-1}(x;q), \quad n\in\N_{0},
\]
and initial conditions  $\theta^{(a)}_{-1}(x;q)=0$ and $\theta^{(a)}_{0}(x;q)=1$, where $a$ is a real parameter. Namely, one arrives at formulas
\begin{equation}
 \tilde{p}_{n}(x)=\frac{(-1)^{n}q^{-n(n-2)/2}}{(1-q)^{n}}\theta_{n}^{\left((q^{-1/2}-q^{1/2})x\right)}\left(\frac{q^{1/2}+q^{-1/2}}{2};q\right)
 \label{eq:p_n_q-cheb}
\end{equation}
and
\begin{equation}
 \tilde{q}_{n}(x)=\frac{(-1)^{n}q^{-n^{2}/2}}{(1-q)^{n}}\theta_{n}^{\left((q^{-1/2}-q^{1/2})qx\right)}\left(\frac{q^{1/2}+q^{-1/2}}{2};q\right).
 \label{eq:q_n_q-cheb}
\end{equation}
In~\cite[Thm.~2.1]{ism-mul_siamjma87}, it was proved that, for $x\notin[-1,1]$, one has the asymptotic formula
\[
  \theta^{(a)}_{n}(x;q)=A^{n}\left[\sum_{k=0}^{\infty}\frac{q^{k(k-1)/2}(-aB)^{k}}{(B/A;q)_{k+1}(q;q)_{k}}\right]\left(1+o(1)\right), \quad n\to\infty,
\]
where $A$ and $B$ are the roots of the quadratic equation $1-2xt+t^{2}=0$ such that $|A|\geq|B|$ and
$(\alpha;q)_{k}=(1-\alpha)(1-\alpha q)\dots(1-\alpha q^{k-1})$ is the $q$-Pochhammer symbol.
If applied in~\eqref{eq:p_n_q-cheb} and \eqref{eq:q_n_q-cheb} and using~\eqref{eq:tow_p_n_ident}, we compute the asymptotic expansion
\begin{align*}
p_{n}(x)&=\frac{(-1)^{n}q^{-n(n-1)/2}}{(1-q)^{n}}\left(\sum_{k=0}^{\infty}\frac{q^{k(k-1)/2}}{(q;q)_{k}^{2}}(1-q)^{k}(-x)^{k}\right)\left[1+o(1)\right]\\
&=\frac{(-1)^{n}q^{-n(n-1)/2}}{(1-q)^{n}}\pPhiq{1}{1}{0}{q}{(1-q)x}\left[1+o(1)\right], \quad n\to\infty,
\end{align*}
where we have used the standard notation for the basic hypergeometric function, see~\cite{gas-rah_04}.
Now, an application of Theorem~\ref{thm:char_func_L-oper} and simple manipulations yield the desired formula:
\begin{equation}
 \det(1-zL)=\pPhiq{1}{1}{0}{q}{(1-q)z}, \quad z\in\C.
\label{eq:char_func_geometric}
\end{equation}

Characteristic function~\eqref{eq:char_func_geometric} can be identified with the third Jackson $q$-Bessel function, also known as the Hahn--Exton $q$-Bessel function, of order zero
\[
 J_{0}^{(3)}(z;q):=\pPhiq{1}{1}{0}{q}{qz^{2}}.
\]
Properties of the third Jackson $q$-Bessel functions were intensively studied in literature~\cite{ann-mas_mpcps09, koe-swa_jmaa94, koe-vanass_ca95, swarttow-phd92, sta-sto_sm13}. Concerning their zeros, they are not known explicitly similarly to the case of classical Bessel functions. Nevertheless, properties of these zeros, that are readily related to the spectrum of $L$, were also investigated in separate works; see, for instance, \cite[Prop.~A.3]{sta-sto_jat16} for an asymptotic formula for the zeros and~\cite{abr-bus-car_ijmms03} for explicit bounds on the zeros.
\end{example}

\section{Spectral analysis of the Hilbert $L$-operator}\label{sec:Hilbert_L-matrix}

Our main goal is a spectral analysis of the Hilbert $L$-operator $L_{\nu}$ determined by the parameter sequence
\[
a_{n}(\nu)=\frac{1}{n+\nu}, \quad n\in\N_{0},
\]
for $\nu\in\R\setminus(-\N_{0})$. The inverse $J_{\nu}:=L_{\nu}^{-1}$ is an unbounded Jacobi operator given by the Jacobi matrix~\eqref{eq:def_J}, where
\[
b_{n}=b_{n}(\nu)=\frac{1}{a_{n}(\nu)-a_{n+1}(\nu)}=(n+\nu)(n+\nu+1), \quad n\in\N_{0}.
\]
Operator $J_{\nu}$ is self-adjoint which is a consequence of claim~(iii) of Lemma~\ref{lem:tridiagonal_inverse}. Moreover, $J_{\nu}$ is semi-bounded as shows the following lemma.

\begin{lem}\label{lem:J_positive}
 For all $\phi\in\spn\{e_{n}\mid n\in\N_{0}\}$, one has
 \[
 \langle\phi,J_{\nu}\phi\rangle= \nu|\phi_{0}|^{2}+\sum_{n=1}^{\infty}\left|(n+\nu)\phi_{n}-(n-1+\nu)\phi_{n-1}\right|^{2}.
 \]
 In particular, $J_{\nu}\geq\min(0,\nu)$.
\end{lem}

\begin{proof}
 Suppose $\phi\in\spn\{e_{n}\mid n\in\N_{0}\}$. Then
 \begin{align*}
   \langle\phi,J_{\nu}\phi\rangle&=\nu(\nu+1)|\phi_{0}|^{2}-\nu(\nu+1)\overline{\phi_{0}}\phi_{1}\\
   &\hskip8pt+\sum_{n=1}^{\infty}\overline{\phi_{n}}\left[-(n-1+\nu)(n+\nu)\phi_{n-1}+2(n+\nu)^{2}\phi_{n}-(n+\nu)(n+1+\nu)\phi_{n+1}\right]\\
   &=\nu|\phi_{0}|^{2}+\sum_{n=1}^{\infty}(n+\nu)\overline{\phi_{n}}\left[(n+\nu)\phi_{n}-(n-1+\nu)\phi_{n-1}\right]\\
   &\hskip144pt-\sum_{n=0}^{\infty}(n+\nu)\overline{\phi_{n}}\left[(n+1+\nu)\phi_{n+1}-(n+\nu)\phi_{n}\right].
 \end{align*}
 By shifting the index in the last sum, we arrive at the equation from the statement.
\end{proof}

Further, we investigate spectral properties of $J_{\nu}$ in more detail. In the analysis, the crucial role is played by the unit argument hypergeometric function
\begin{equation}
 \pFq{3}{2}{a_{1},a_{2},a_{3}}{b_{1},b_{2}}{1}:=\sum_{n=0}^{\infty}\frac{(a_{1})_{n}(a_{2})_{n}(a_{3})_{n}}{n!\,(b_{1})_{n}(b_{2})_{n}},
\label{eq:def_3F2}
\end{equation}
where $(a)_{n}=a(a+1)\dots(a+n-1)$ is the Pochhammer symbol. Function~\eqref{eq:def_3F2} is well defined, if $\Re(b_{1}+b_{2}-a_{1}-a_{2}-a_{3})>0$ and $b_{1},b_{2}\notin-\N_{0}$. The latter condition can be removed by introducing the regularized hypergeometric function 
\[
 \regpFq{3}{2}{a_{1},a_{2},a_{3}}{b_{1},b_{2}}{1}:=\frac{1}{\Gamma(b_{1})\Gamma(b_{2})}\,\pFq{3}{2}{a_{1},a_{2},a_{3}}{b_{1},b_{2}}{1},
\]
which is analytic in each parameter $a_{1},a_{2},a_{3},b_{1},b_{2}$, if $\Re(b_{1}+b_{2}-a_{1}-a_{2}-a_{3})>0$.

At first, we show that, in the special case $\nu=1$, Jacobi operator $J_{1}$ corresponds to a~subfamily of the Continuous dual Hahn polynomials whose properties are well known. As a result, the spectral analysis of~$J_{1}$ readily follows. In the general case, however, one cannot make a direct use of existing results although orthogonal polynomials determined by $J_{\nu}$ can be found as a particular case of the five-parameter family studied by Ismail, Letessier, and Valent in~\cite{ism-let-val_siamjma89} within the context of birth and death processes with quadratic rates. However, in our particular setting, the parameters are out of range of the analysis made in~\cite{ism-let-val_siamjma89}.

\begin{rem}
Let us explain the last comment in more detail. If we use the notation of~\cite{ism-let-val_siamjma89} and set
\begin{equation}
a:=\nu, \quad b:=\nu+1, \quad \alpha:=\nu-1, \quad \beta:=\nu, \quad\mbox{ and }\quad \eta:=0,
\label{eq:parameters_ismail-etal}
\end{equation}
then polynomials $P_{n}(x)=P_{n}(x;a,b,\alpha,\beta,\eta)$ studied in~\cite{ism-let-val_siamjma89} coincide with orthogonal polynomials determined by $J_{\nu}$ up to a multiplicative constant. However, the approach to derive an asymptotic formula for $P_{n}(x)$, as $n\to\infty$, used in~\cite{ism-let-val_siamjma89} does not apply in the particular case of parameters~\eqref{eq:parameters_ismail-etal}, which is not immediately obvious from the paper. The asymptotic analysis relies on an application of the Darboux method to a convenient generating function of the sequence $P_{n}(x)$, which is, in its turn, found as a solution of a hypergeometric differential equation, see~\cite[Eqs.~(2.1) and~(2.2)]{ism-let-val_siamjma89}. In the exceptional case of parameters~\eqref{eq:parameters_ismail-etal}, however, functions $G_{1}^{\alpha,\beta}$ and $G_{2}^{\alpha,\beta}$ defined in~\cite[p.~729]{ism-let-val_siamjma89} do not form a fundamental system of the respective ODE, see~\cite[\S~15.10]{dlmf} for more details. As a result, the generating function is not given by~\cite[Eq.~(2.6)]{ism-let-val_siamjma89}. From this point of view, the spectral analysis of $J_{\nu}$ is of independent interest since it complements the picture on the corresponding family of orthogonal polynomials.
\end{rem}

\subsection{A warm up: Special case $\nu=1$ and corollaries}\label{subsec:nu=1}

Recall that the Continuous dual Hahn polynomials $S_{n}(x;a,b,c)$ is a three-parameter family of hypergeometric orthogonal polynomials listed in the Askey scheme~\cite[Sec.~9.3]{koe-les-swa_10}. They are defined by the formula~\cite[Eq.~9.3.1]{koe-les-swa_10}
\[
S_{n}(x^{2};a,b,c):=(a+b)_{n}(a+c)_{n}\,\pFq{3}{2}{-n,a+\ii x,a-\ii x}{a+b,a+c}{1}.
\]
In the particular case $a=b=1/2$, $c=3/2$, the sequence of polynomials 
\[
P_{n}(x):=\frac{1}{n!(n+1)!}\,S_{n}\!\left(x-\frac{1}{4};\frac{1}{2},\frac{1}{2},\frac{3}{2}\right), \quad n\in \N_{0},
\]
fulfills the recurrence
\begin{align*}
 (b_{0}(1)-x)P_{0}(x)-b_{0}(1)P_{1}(x)&=0,\\
-b_{n-1}(1)P_{n-1}(x)+(b_{n-1}(1)+b_{n}(1)-x)P_{n}(x)-b_{n}(1)P_{n+1}(x)&=0, \quad n\in\N,
\end{align*}
which follows from the general three-term recurrence for Continuous dual Hahn polynomials~\cite[Eq.~(9.3.4)]{koe-les-swa_10}. In other words, the column vector $P(x)=(P_{0}(x),P_{1}(x),\dots)^{T}$ fulfills the generalized eigenvalue equation $J_{1}P(x)=xP(x)$. Moreover, we have the orthogonality relation
\begin{equation}
2\pi\int_{0}^{\infty}P_{m}\!\left(\frac{1}{4}+x^{2}\right)P_{n}\!\left(\frac{1}{4}+x^{2}\right)\frac{x\sinh(\pi x)}{\cosh^{2}(\pi x)}\!\left(\frac{1}{4}+x^{2}\right)\!\dd x=\delta_{m,n}, \quad m,n\in\N_{0},
\label{eq:orthogonality_nu=1}
\end{equation}
which can be deduced from the general orthogonality relation~\cite[Eq.~(9.3.2)]{koe-les-swa_10} and well-known special values of the Gamma function, namely~\cite[Eqs.~5.4.3, 5.4.4]{dlmf}.

This means that $J_{1}$ is unitarily equivalent to the operator of multiplication by the independent variable on $\mathscr{H}:=L^{2}([1/4,\infty),\rho(x)\dd x)$, where 
\[
\rho(x)=\pi x\frac{\sinh(\pi\sqrt{x-1/4})}{\cosh^{2}(\pi\sqrt{x-1/4})},
\]
via the unitary mapping $U:\ell^{2}(\N_{0})\to \mathscr{H}$ unambiguously determined by correspondence $U:e_{n}\mapsto P_{n}$, $\forall n\in\N_{0}$. Consequently, $\sigma(J_{1})=\sigma_{\text{ac}}(J_{1})=[1/4,\infty)$ and the spectrum is simple (which is always the case for self-adjoint Jacobi operators). Since $L_{1}=J_{1}^{-1}$ we immediately obtain the following statement.

\begin{prop}\label{prop:L_1_spectrum}
The spectrum of $L_{1}$ is simple, purely absolutely continuous, and equal to the interval $[0,4]$.
\end{prop}

Last observation gives us also an information about the absolutely continuous spectrum of $J_{\nu}$, for arbitrary $\nu\in\R\setminus(-\N_{0})$. We will also arrive at the same result by direct means below.

\begin{prop}
For all $\nu\in\R\setminus(-\N_{0})$, the absolutely continuous spectrum of $L_{\nu}$ is simple and fills the interval $[0,4]$.
\end{prop}

\begin{proof}
 Let $\nu\in\R\setminus(-\N_{0})$ be fixed. 
 Since $L_{1}-L_{\nu}$ is an $L$-operator with the parameter sequence
 \[
  a_{n}(1)-a_{n}(\nu)=\frac{\nu-1}{(n+1)(n+\nu)}, \quad n\in\N_{0},
 \]
 it is a trace class operator by Theorem~\ref{thm:l-matrix_trace_class}. The statement now follows from Proposition~\ref{prop:L_1_spectrum} and the Kato--Rosenblum Theorem, which implies that the absolutely continuous parts of $L_{1}$ and $L_{\nu}$ are unitarily equivalent, see~\cite[Thm.~4.4, p.~540]{kat_66}.
\end{proof}

\subsection{Solutions of the eigenvalue equation $J_{\nu}$}
\label{subsec:eigvectors_of_J_nu}

At the start, we define a sequence of functions that will be of importance in the spectral analysis of $J_{\nu}$ for general $\nu\in\R\setminus(-\N_{0})$. Its definition for the full range of the involved parameters requires a certain cautiousness. 

First, for $n\in\N_{0}$, $n+\nu>0$, and $z\in\C$,  we define
\begin{equation}
\phi_{n}(z;\nu):=\Gamma(n+\nu+1)\,\regpFq{3}{2}{1/2+z, 1/2+z, 3/2+z}{1+2z,n+\nu+3/2+z}{1}.
\label{eq:def_phi}
\end{equation}
Clearly, the condition $n+\nu>0$ is fulfilled for $n$ sufficiently large and, if $\nu>0$, for all $n\in\N_{0}$. However, we wish to give meaning to the right hand side of~\eqref{eq:def_phi} also if $n+\nu<0$. One way to do so, is to define $\phi_{n}$ recursively for $n<-\nu$ by using the following recurrence rule.

\begin{lem}\label{lem:recur_phi}
 For all $n\in\N$, $\nu>0$, and $z\in\C$, one has
 \[
  -(n-1+\nu)(n+\nu)\phi_{n-1}(z;\nu)+\left[2(n+\nu)^{2}-\frac{1}{4}+z^{2}\right]\phi_{n}(z;\nu)-(n+\nu)(n+1+\nu)\phi_{n+1}(z;\nu)=0.
 \]
\end{lem}

\begin{proof}
We deduce the recurrence formula from a general identity for the ${}_{3}F_{2}$-functions. If we temporarily denote 
\[
 F_{i}:=\frac{1}{\Gamma(b_{1}+i)}\,\pFq{3}{2}{a_{1},a_{2},a_{3}}{b_{1}+i,b_{2}}{w},
\]
then the identity
\begin{equation}
(w-1)F_{-1}=(A_{1}+wB_{1})F_{0}+(A_{2}+wB_{2})F_{1}+wB_{3}F_{2},
\label{eq:recur_G_general}
\end{equation}
holds true for 
\begin{align*}
A_{1}&=b_{2}-2b_{1},\\
B_{1}&=3b_{1}-a_{1}-a_{2}-a_{3},\\
A_{2}&=b_{1}(b_{1}-b_{2}+1),\\
B_{2}&=(2b_{1}+1)(a_{1}+a_{2}+a_{3})-3b_{1}^{2}-3b_{1}-1-a_{1}a_{2}-a_{1}a_{3}-a_{2}a_{3},\\
B_{3}&=(b_{1}-a_{1}+1)(b_{1}-a_{2}+1)(b_{1}-a_{3}+1).
\end{align*}
This is one of several contiguous relations for ${}_{3}F_{2}$-functions. Recall that any four distinct contiguous functions ${}_{3}F_{2}$, i.e., ${}_{3}F_{2}$-functions whose parameters differ by integers, are linearly related. We do not have an exact reference for the particular identity~\eqref{eq:recur_G_general} but it can be verified straightforwardly by comparing coefficients of the same powers of $w$ on both sides of~\eqref{eq:recur_G_general}; see also~\cite[\S~48]{rai_71}.

To obtain the formula from the statement, it suffices to set $w=1$ and
\[
 a_{1}=a_{2}=\frac{1}{2}+z, \quad a_{3}=\frac{3}{2}+z,\quad b_{1}=n+\nu+\frac{1}{2}+z,\quad b_{2}=1+2z
\]
in~\eqref{eq:recur_G_general} and do simple manipulations.
\end{proof}

With the aid of Lemma~\ref{lem:recur_phi}, we can extend the definition of $\phi_{n}(z;\nu)$ recursively for $n<-\nu$ by equation
\[
  \phi_{n}(z;\nu)=\frac{2(n+1+\nu)^{2}-1/4+z^{2}}{(n+\nu)(n+1+\nu)}\phi_{n+1}(z;\nu)-\frac{n+2+\nu}{n+\nu}\phi_{n+2}(z;\nu).
\]
This way, $\phi_{n}(z;\nu)$ is defined for all $n\in\Z$ and analytic in $z\in\C$ and $\nu\in\C\setminus\{-n,-n-1,\dots\}$.

It is useful to have a relatively simple hypergeometric formula like~\eqref{eq:def_phi} for $\phi_{n}(z;\nu)$ rather than the recursive definition. This is possible, however, the range of parameter $z$ is to be restricted to a half-plane. We make use of the Thomae transformation~\cite[Eq.~7.4.4.2]{prudnikov-etal_vol3}
\begin{equation}
 \regpFq{3}{2}{a,b,c}{d,e}{1}=\frac{\Gamma(d+e-a-b-c)}{\Gamma(a)}\,\regpFq{3}{2}{d-a,e-a,d+e-a-b-c}{d+e-a-b,d+e-a-c}{1},
\label{eq:prudnikov_transf2}
\end{equation}
which holds true if $\Re(d+e-a-b-c)>0$ and $\Re a>0$. By putting
\[
 a=\frac{3}{2}+z, \quad b=c=\frac{1}{2}+z,\quad d=1+2z,\quad e=n+\nu+\frac{3}{2}+z,
\]
we obtain from~\eqref{eq:def_phi} the equality
\begin{equation}
 \phi_{n}(z;\nu)=\frac{\Gamma(n+\nu)\Gamma(n+\nu+1)}{\Gamma(z+3/2)}\,\regpFq{3}{2}{z-1/2,n+\nu,n+\nu}{n+\nu+z+1/2,n+\nu+z+1/2}{1},
\label{eq:def_phi_extended}
\end{equation}
provided that $n+\nu>0$ and $\Re z>-3/2$. However, the condition $n+\nu>0$ is only necessary for the hypergeometric function ${}_{3}\tilde{F}_{2}$ present in definition~\eqref{eq:def_phi} to be well defined. Note that, for $n\in\Z$ and $z\in\C$, $\Re z>-3/2$ fixed, the right-hand side of~\eqref{eq:def_phi_extended} is an analytic function in $\nu\in\C\setminus\{-n,-n-1,\dots\}$. Therefore the right-hand side of~\eqref{eq:def_phi_extended} has to coincide with the recursively extended function $\phi_{n}(z;\nu)$ for all $n\in\Z$, $\nu\in\C\setminus\{-n,-n-1,\dots\}$, and $z\in\C$ such that $\Re z>-3/2$.

One could observe already from Lemma~\ref{lem:recur_phi} that functions $\phi_{n}(z;\nu)$ are almost generalized eigenvectors of the Jacobi operator $J_{\nu}$, therefore they play a crucial role in the spectral analysis of $J_{\nu}$. Their selected properties are summarized in the following proposition.

\begin{prop}\label{prop:phi_fund_proper}
 Let $\nu\in\R\setminus(-\N_{0})$ and $z\in\C$. 
 \begin{enumerate}[1)]
 \item The semi-infinite column vector $\phi(z;\nu)=(\phi_{0}(z;\nu),\phi_{1}(z;\nu),\dots)^{T}$ fulfills the equation
 \[
  J_{\nu}\phi(z;\nu)=\left(\frac{1}{4}-z^{2}\right)\phi(z;\nu)+\chi(z;\nu)e_{0},
 \]
 where $\chi(z;\nu):=\nu(\nu-1)\left[\phi_{-1}(z;\nu)-\phi_{0}(z;\nu)\right]$. In addition, we have the hypergeometric representation 
 \begin{align}
  \chi(z;\nu)=\frac{\Gamma(\nu)\Gamma(\nu+1)}{\Gamma(z+3/2)}\bigg[&\regpFq{3}{2}{z-1/2,\nu-1,\nu-1}{\nu+z-1/2,\nu+z-1/2}{1}\nonumber\\
  &\hskip18pt-\nu(\nu-1)\,\regpFq{3}{2}{z-1/2,\nu,\nu}{\nu+z+1/2,\nu+z+1/2}{1}\bigg],
 \label{eq:chi_hyper_form}
 \end{align}
 provided that $\Re z>-3/2$.
 \item For $n\to\infty$, the asymptotic expansion
 \[
  \phi_{n}(z;\nu)=\frac{n^{-z-1/2}}{\Gamma(1+2z)}\left[1+O\left(\frac{1}{n}\right)\right],
 \]
 holds, provided that $-2z\notin\N$.
 \item The Wronskian  
 \[
 W(\phi(z;\nu),\phi(-z;\nu)):=-b_{n}(\nu)\left[\phi_{n+1}(z;\nu)\phi_{n}(-z;\nu)-\phi_{n+1}(-z;\nu)\phi_{n}(z;\nu)\right]
 \] 
 of vectors $\phi(\pm z;\nu)$ reads
 \[
  W(\phi(z;\nu),\phi(-z;\nu))=\frac{\sin(2\pi z)}{\pi}.
 \]
 Consequently, $\phi(z;\nu)$ and $\phi(-z;\nu)$ are linearly independent if and only if $2z\notin\Z$.
\end{enumerate}  
\end{prop}

\begin{rem}
More expressions for $\chi(z;\nu)$, even simpler than~\eqref{eq:chi_hyper_form}, are given in Proposition~\ref{prop:chi_expressions} below.
\end{rem}

\begin{proof}[Proof of Proposition~\ref{prop:phi_fund_proper}]
 1) The first claim follows readily from Lemma~\ref{lem:recur_phi} and formula~\eqref{eq:chi_hyper_form} is obtained from~\eqref{eq:def_phi_extended}.
 
 2) For $\nu\in\R\setminus\N_{0}$, $z\in\C\setminus\left(-\N/2\right)$ fixed and $n$ large, we have formula~\eqref{eq:def_phi}, which is in a suitable form for the asymptotic analysis. First, from the very definition of the hypergeometric series~\eqref{eq:def_3F2}, we have
 \[
 \pFq{3}{2}{1/2+z, 1/2+z, 3/2+z}{1+2z,n+\nu+3/2+z}{1}=1+O\left(\frac{1}{n}\right), \quad n\to\infty.
 \]
 Second, the Stirling formula yields
 \[
  \frac{\Gamma(n+\nu+1)}{\Gamma\left(n+\nu+3/2+z\right)}=n^{-z-1/2}\left[1+O\left(\frac{1}{n}\right)\right], \quad n\to\infty.
 \]
 Applying the two asymptotic formulas in~\eqref{eq:def_phi}, we arrive at the statement.
 
 3) Since both vectors $\phi(\pm z;\nu)$ solve the second order difference equation from Lemma~\ref{lem:recur_phi} their Wronskian $W(\phi(z;\nu),\phi(-z;\nu))$ is an $n$-independent constant, which vanishes if and only if $\phi(z;\nu)$ and $\phi(-z;\nu)$ are linearly dependent. Therefore we can compute the Wronskian with the aid of the already proven claim 2). We have
 \begin{align*}
  W(\phi(z;\nu),\phi(-z;\nu))&=\lim_{n\to\infty}(n+\nu)(n+\nu+1)\,\frac{n^{-z-1/2}(n+1)^{z-1/2}-n^{z-1/2}(n+1)^{-z-1/2}}{\Gamma(1+2z)\Gamma(1-2z)}\\
  &=\frac{2z}{\Gamma(1+2z)\Gamma(1-2z)}=\frac{1}{\Gamma(2z)\Gamma(1-2z)}=\frac{\sin(2\pi z)}{\pi},
 \end{align*}
 for $2z\notin\Z$, where we have used the multiplication formula~\cite[Eq.~5.5.3]{dlmf} 
 \[ 
 \Gamma(z)\Gamma(1-z)=\frac{\pi}{\sin(\pi z)}.
 \] 
  By definition, the function $z\mapsto W(\phi(z;\nu),\phi(-z;\nu))$ is entire and hence the obtained formula extends to all $z\in\C$.
\end{proof}

Function $\chi(z;\nu)$ defined in claim 1) of Proposition~\ref{prop:phi_fund_proper} will play a role of the characteristic function of $J_{\nu}$. Three more hypergeometric expressions for $\chi(z;\nu)$, which will be needed below, are given in the next statement.

\begin{prop}\label{prop:chi_expressions}
The following hypergeometric representations hold:
\begin{enumerate}[1)]
\item For $\nu>0$ and $z\in\C$, one has
 \begin{equation}
  \chi(z;\nu)=\left(z+\frac{1}{2}\right)\Gamma(\nu+1)\,\regpFq{3}{2}{z-1/2,z+1/2,z+3/2}{2z+1,z+\nu+1/2}{1}.\label{eq:chi_simple_hyp_form1}
 \end{equation}
\item For $\nu\in\R\setminus(-\N_{0})$ and $\Re z>-1/2$, one has
  \begin{equation}
   \chi(z;\nu)=\frac{(z+1/2)\Gamma(\nu)\Gamma(\nu+1)}{\Gamma(z+1/2)}\,\regpFq{3}{2}{\nu-1,\nu+1,z+1/2}{z+\nu+1/2,z+\nu+1/2}{1}.
   \label{eq:chi_simple_hyp_form2}
   \end{equation}
\item For $\nu\in\R\setminus(-\N_{0})$ and $\Re z>-1/2$, one has
 \begin{equation}
   \chi(z;\nu)=\frac{(z+1/2)\Gamma(\nu)\Gamma(\nu+1)}{\Gamma(z+1/2)}\,\regpFq{3}{2}{\nu,\nu,z+1/2}{z+\nu-1/2,z+\nu+3/2}{1}.
  \label{eq:chi_simple_hyp_form3}
  \end{equation}
  \end{enumerate}
\end{prop}

\begin{proof}
1) First, applying the transformation~\cite[Eq.~7.4.4.1]{prudnikov-etal_vol3}
\begin{equation}
 \regpFq{3}{2}{a,b,c}{d,e}{1}=\frac{\Gamma(d+e-a-b-c)}{\Gamma(d-c)}\,\regpFq{3}{2}{e-a,e-b,c}{d+e-a-b,e}{1},
\label{eq:prudnikov_transf1}
\end{equation}
which holds true if $\Re(d+e-a-b-c)>0$ and $\Re(d-c)>0$, in~\eqref{eq:chi_hyper_form}, we obtain
\begin{align*}
 \chi(z;\nu)&=\Gamma(\nu+1)\bigg[\regpFq{3}{2}{z-1/2,z+1/2,z+1/2}{2z+1,z+\nu-1/2}{1}\\
 &\hskip139pt-(\nu-1)\,\regpFq{3}{2}{z-1/2,z+1/2,z+1/2}{2z+1,z+\nu+1/2}{1}\bigg]\\
&=\Gamma(\nu+1)\sum_{k=0}^{\infty}\frac{(z-1/2)_{k}(z+1/2)_{k}^{2}}{k!\,\Gamma(2z+1+k)\Gamma(z+\nu+1/2+k)}\left[z+\nu-\frac{1}{2}+k-(\nu-1)\right]\\
&=\Gamma(\nu+1)\sum_{k=0}^{\infty}\frac{(z-1/2)_{k}(z+1/2)_{k}(z+1/2)_{k+1}}{k!\,\Gamma(2z+1+k)\Gamma(z+\nu+1/2+k)},
\end{align*}
for $\nu>0$ and $\Re z>-3/2$. This yields formula~\eqref{eq:chi_simple_hyp_form1} for $\nu>0$ and $\Re z>-3/2$. Recall $\chi(z;\nu)$ is analytic in $z\in\C$ and $\nu\in\R\setminus(-\N_{0})$. Since the right-hand side of~\eqref{eq:chi_simple_hyp_form1} is a well defined function for all $\nu>0$ and entire in $z$, the formula~\eqref{eq:chi_simple_hyp_form1} extends to all $z\in\C$ by analyticity.
 
2) If we apply~\eqref{eq:prudnikov_transf1} once more this time in~\eqref{eq:chi_simple_hyp_form1} with
\[
 a=z-\frac{1}{2}, \quad b=z+\frac{3}{2}, \quad c=z+\frac{1}{2}, \quad d=2z+1, \quad e=z+\nu+\frac{1}{2},
\]
we establish~\eqref{eq:chi_simple_hyp_form2} for $\nu>0$ and $\Re z>-1/2$. Since the right-hand side of~\eqref{eq:chi_simple_hyp_form2} is analytic in $\nu\in\R\setminus(-\N_{0})$, the formula extends by analyticity.

3) Finally, the Thomae transformation~\eqref{eq:prudnikov_transf2} used in~\eqref{eq:chi_simple_hyp_form1} with
\[
 a=z+\frac{1}{2}, \quad b=z-\frac{1}{2}, \quad c=z+\frac{3}{2}, \quad d=2z+1, \quad e=z+\nu+\frac{1}{2},
\]
yields~\eqref{eq:chi_simple_hyp_form3} for $\nu>0$ and $\Re z>-1/2$ which extends again to all $\nu\in\R\setminus(-\N_{0})$ by analyticity.
\end{proof}

\subsection{Spectral analysis of $J_{\nu}$}
\label{subsec:spectrum_of_J_nu}

With the aid of Proposition~\ref{prop:phi_fund_proper}, we may determine the point spectrum of $J_{\nu}$.

\begin{prop}\label{prop:point_spec_J_nu}
 For all $\nu\in\R\setminus(-\N_{0})$, we have
 \[
  \sigma_{\text{p}}(J_\nu)=\left\{\frac{1}{4}-x^{2} \;\bigg|\; \chi(x;\nu)=0, \; x>0\right\}.
 \]
 In addition, $0\notin\sigma_{\text{p}}(J_{\nu})$.
\end{prop}

\begin{proof}
 We write the spectral parameter $\lambda\in\R$ in the form $\lambda=1/4-z^{2}$, for $z\in\R\cup\ii\R$. In fact, without loss of generality, it suffices to consider the following four cases:
 \[
  i)\, z=\ii y,\, y>0, \qquad ii)\,z=x,\, x>0, \qquad iii)\, z=0, \qquad iv)\, z=\frac{1}{2}.
 \] 
 
 Ad i) By Proposition~\ref{prop:phi_fund_proper}, vectors $\phi(\pm\ii y;\nu)$ are linearly independent for all $y>0$ and their linear combination 
 \[
  \psi(y;\nu):=\chi(-\ii y;\nu)\phi(\ii y;\nu)-\chi(\ii y;\nu)\phi(-\ii y;\nu)
 \]
 is a solution of the eigenvalue equation $J_{\nu}\psi(y;\nu)=\left(1/4+y^{2}\right)\psi(y;\nu)$ which is determined uniquely up to a multiplicative constant. Vector $\psi(y;\nu)$ is nontrivial since 
 \begin{align*}
  \psi_{0}(y;\nu)&=\nu(\nu-1)\left[\phi_{-1}(-\ii y;\nu)\phi_{0}(\ii y;\nu)-\phi_{-1}(\ii y;\nu)\phi_{0}(-\ii y;\nu)\right]=W\left(\phi(-\ii y;\nu),\phi(\ii y;\nu)\right)\\
  &=\frac{\sinh(2\pi y)}{\ii\pi}\neq0,
 \end{align*}
 for $y>0$. In other words, the coefficients $\chi(\pm\ii y;\nu)$ cannot vanish simultaneously. Moreover, it follows from claim~2) of Proposition~\ref{prop:phi_fund_proper} that
 \[
  \psi_{n}(y;\nu)=\frac{1}{\sqrt{n}}\left[\frac{\chi(-\ii y;\nu)}{\Gamma(1+2\ii y)}\,n^{-\ii y}-\frac{\chi(\ii y;\nu)}{\Gamma(1-2\ii y)}\,n^{\ii y}\right]\!\left(1+O\left(\frac{1}{n}\right)\right)\!, \quad n\to\infty.
 \]
 Hence $\psi(y;\nu)\notin\ell^{2}(\N_{0})$, for all $y>0$. Consequently, $(1/4,\infty)\cap\sigma_{\p}(J_{\nu})=\emptyset$.
 
 Ad ii) By the asymptotic expansion from claim 2) of Proposition~\ref{prop:phi_fund_proper}, $\phi(x;\nu)$ is non-trivial and square summable for all $x>0$. Moreover, we have $J_{\nu}\phi(x;\nu)=\left(1/4-x^{2}\right)\phi(x;\nu)$ if and only if $\chi(x;\nu)=0$. Thus, $1/4-x^{2}\in\sigma_{\p}(J_{\nu})$, for $x>0$, if and only if $\chi(x;\nu)=0$.

 Ad iii) Next, we verify that $1/4\notin\sigma_{\p}(J_\nu)$. Note that Proposition~\ref{prop:phi_fund_proper} does not provide us with two linearly independent solutions of the difference equation
 \begin{equation}
 -(n-1+\nu)(n+\nu)\phi_{n-1}+\left[2(n+\nu)^{2}-\frac{1}{4}\right]\phi_{n}-(n+\nu)(n+1+\nu)\phi_{n+1}=0,
 \label{eq:diff_eq_z=0}
 \end{equation}
 but only one $\phi=\phi(0;\nu)$. The second solution of~\eqref{eq:diff_eq_z=0} can be identified in terms of the hypergeometric functions, too, namely
 \[
  \xi_{n}=\Gamma(-n-\nu)\Gamma(-n-\nu+1)\,\regpFq{3}{2}{-1/2,-n-\nu,-n-\nu}{-n-\nu+1/2,-n-\nu+1/2}{1}, \quad n\in\N_{0}.
 \]
 In contrast to $\phi_{n}$, a direct asymptotic analysis of solution $\xi_{n}$, for $n\to\infty$, is a difficult task. However, in order to verify that $1/4\notin\sigma_{\p}(J_\nu)$, an explicit asymptotic formula is not needed. At this point, it is more straightforward to apply the method of successive approximations to the difference equation~\eqref{eq:diff_eq_z=0}; see~\cite{li-wong_jcam92}.
 
 The method of successive approximations implies that there are two linearly independent solutions of~\eqref{eq:diff_eq_z=0} with the asymptotic behavior
 \begin{equation}
  \phi_{n}=\frac{1}{\sqrt{n}}\left[1+O\left(\frac{1}{n}\right)\right] \quad\mbox{ and }\quad \xi_{n}=\frac{\log n}{\sqrt{n}}\left[1+O\left(\frac{1}{n}\right)\right],
 \label{eq:sol_asympt_z=0}
 \end{equation}
 for $n\to\infty$; more details are given in Appendix~\ref{subsec:a.1} for reader's convenience. Now, it is obvious from~\eqref{eq:sol_asympt_z=0} that there is no nontrivial linear combination of $\phi$ and $\xi$, which belongs to $\ell^{2}(\N_{0})$. It means that $1/4\notin\sigma_{\text{p}}(J_\nu)$.
 
  Ad iv) The claim $0\notin\sigma_{\p}(J_{\nu})$ is a general fact, which follows from Lemma~\ref{lem:tridiagonal_inverse}. Nevertheless, we can also check that $\chi(1/2;\nu)\neq0$ by a direct computation. To this end, one uses~\eqref{eq:chi_hyper_form} getting
  \[
  \chi\left(\frac{1}{2};\nu\right)=\Gamma(\nu)\Gamma(\nu+1)\left[\frac{1}{\Gamma^{2}(\nu)}-\frac{\nu(\nu-1)}{\Gamma^{2}(\nu+1)}\right]=1.
  \]
\end{proof}

Our next goal is to compute the Weyl $m$-function of $J_{\nu}$, 
\[
 m(\lambda;\nu):=\langle e_{0}, (J_{\nu}-\lambda)^{-1}e_{0}\rangle,
\]
for $\lambda\in\C\setminus\R$, and deduce the spectral measure $\mu_{\nu}:=\langle e_{0},E_{J_{\nu}}e_{0}\rangle$, where $E_{J_{\nu}}$ is the projection-valued spectral measure of the self-adjoint operator $J_{\nu}$. We refer the reader to~\cite{akh_65,tes_00} for general theory of Jacobi operators. In order to deduce a formula for the Weyl $m$-function, we make use of the formula for a general formula for the Green kernel of a Jacobi operator, which, if applied to~$J_{\nu}$, reads
\begin{equation}
 \langle e_{m}, (J_{\nu}-\lambda)^{-1}e_{n}\rangle=\frac{1}{W(\phi,\psi)}
 \begin{cases}
   \psi_{m}\phi_{n},& \quad\mbox{ if } m\leq n,\\
   \psi_{n}\phi_{m},& \quad\mbox{ if } n\leq m,
 \end{cases}
\label{eq:green_kernel}
\end{equation}
where $\psi$ is a generalized solution to the eigenvalue equation $J_{\nu}\psi=\lambda\psi$, and $\phi$ is a square summable solution of the recurrence from Lemma~\ref{lem:recur_phi} with $\lambda=1/4-z^{2}$.

\begin{prop}\label{prop:m-func}
 Suppose $\nu\in\R\setminus(-\N_{0})$, $\lambda=1/4-z^{2}$, and $\Re z>0, \Im z>0$. Then we have
 \[
  m(\lambda;\nu)=\frac{\phi_{0}(z;\nu)}{\chi(z;\nu)}.
 \]
\end{prop}

\begin{proof}
Note that $\Im\lambda<0$ for $\Re z>0, \Im z>0$.
It follows from Proposition~\ref{prop:phi_fund_proper} that the vector
 \[
  \psi(z;\nu):=\chi(-z;\nu)\phi(z;\nu)-\chi(z;\nu)\phi(-z;\nu)
 \]
 solves the eigenvalue equation 
 $J_{\nu}\psi(z;\nu)=\lambda\psi(z;\nu)$ and also that $\phi(z;\nu)$, the solution of the recurrence from Lemma~\ref{lem:recur_phi}, belongs to $\ell^{2}(\N_{0})$ since we suppose $\Re z>0$. Notice further that 
 \[
  \psi_{0}(z;\nu)=-W(\phi(z;\nu),\phi(-z;\nu)).
 \]
 Hence
 \[
 W(\phi(z;\nu),\psi(z;\nu))=-\chi(z;\nu)W(\phi(z;\nu),\phi(-z;\nu))=\chi(z;\nu)\psi_{0}(z;\nu),
 \]
 and by general formula~\eqref{eq:green_kernel} for the Green kernel of $J_{\nu}$, we obtain
 \[
  m(\lambda;\nu)=\frac{\psi_{0}(z;\nu)\phi_{0}(z;\nu)}{W(\phi(z;\nu),\psi(z;\nu))}=\frac{\phi_{0}(z;\nu)}{\chi(z;\nu)}.
 \]
\end{proof}

\begin{rem}
 Of course, we can combine Proposition~\ref{prop:m-func} with formula~\eqref{eq:def_phi_extended} for $\phi_{0}(z;\nu)$ and any expression from~\eqref{eq:chi_hyper_form}, \eqref{eq:chi_simple_hyp_form1}, \eqref{eq:chi_simple_hyp_form2}, or \eqref{eq:chi_simple_hyp_form3} to express $m(\,\cdot\,;\nu)$ as a ratio of hypergeometric functions. For example, for $\nu\in\R\setminus(-\N_{0})$ and $\Re z>0, \Im z>0$, we have
\[
 m(\lambda;\nu)=\frac{4}{(2z+1)^{2}}\frac{\displaystyle\regpFq{3}{2}{z-1/2,\nu,\nu}{\nu+z+1/2,\nu+z+1/2}{1}}{\displaystyle\regpFq{3}{2}{z+1/2,\nu,\nu}{\nu+z-1/2,\nu+z+3/2}{1}},
\] 
where $\lambda=1/4-z^{2}$.
\end{rem}

The Weyl $m$-function coincides with the Cauchy (or Stieltjes) transform of the spectral measure, i.e., 
\[
 m(\lambda;\nu)=\int_{\R}\frac{\dd\mu_{\nu}(x)}{x-\lambda}, \quad \lambda\in\C\setminus\R. 
\]
Recall the spectral measure $\mu_{\nu}$ is completely determined by its Cauchy transform $m(\,\cdot\,;\nu)$ given on $\C\setminus\R$. In particular, $\supp\mu_{\nu}$ coincides with the domain of analyticity of $m(\,\cdot\,;\nu)$. By Proposition~\ref{prop:m-func}, we have a formula for $m(\lambda;\nu)$ if $\Im\lambda<0$.
By self-adjointness of $J_{\nu}$, we have the symmetry 
\[
\overline{m(\lambda;\nu)}=m(\overline{\lambda};\nu), \quad \lambda\in\C\setminus\R,
\]
which means that $m(\lambda;\nu)$ is determined by Proposition~\ref{prop:m-func} on the half-plane $\Im\lambda>0$, too.

\begin{prop}\label{prop:spectral_meas_J_nu}
 For all $\nu\in\R\setminus(-\N_{0})$, the spectral measure $\mu_{\nu}=\mu_{\nu}^{d}+\mu_{\nu}^{ac}$, where $\mu_{\nu}^{d}$ and $\mu_{\nu}^{ac}$ is the discrete and the absolutely continuous part of $\mu_{\nu}$, respectively. 
  \begin{enumerate}[1)]
   \item Discrete measure $\mu_{\nu}^{d}$ is supported on the set positive zeros of  $\chi(\cdot;\nu)$  	and
   \[
    \mu_{\nu}^{d}=\sum_{\substack{x_{0}>0 \\ \chi(x_{0};\nu)=0}}\frac{\phi_{0}(x_{0};\nu)}{\partial_{x}\chi(x_{0};\nu)}\,\delta_{1/4-x_{0}^{2}},
   \]
   where $\delta_{y_{0}}$ denotes the unit-mass Dirac measure supported on the one-point set $\{y_{0}\}$.
   \item Absolutely continuous measure $\mu_{\nu}^{ac}$ is supported on $[1/4,\infty)$ and its density reads
 \[
  \frac{\dd\mu_{\nu}}{\dd x}\!\left(\frac{1}{4}+x^{2}\right)=\frac{\sinh(2\pi x)}{2\pi^{2}\left|\chi(\ii x;\nu)\right|^{2}},
 \]
 for $x>0$.
  \end{enumerate}
\end{prop}

\begin{proof}
We write again $\lambda=1/4-z^{2}$ and investigate singular points of $m(\lambda;\nu)$ for $\lambda\in\R$. Clearly, $\lambda=\lambda(z)$ maps the positive imaginary line onto $(1/4,\infty)$ and the positive real line onto $(-\infty,1/4)$. 
 
 First, by inspection of the formula $m(\lambda;\nu)=\phi_{0}(z;\nu)/\chi(z;\nu)$, for $z>0$, one finds that the only singular points of $m(\,\cdot\,;\nu)$ located in $(-\infty,1/4)$ are simple poles determined by zeros of $\chi(\,\cdot\,;\nu)$ in $(0,\infty)$. Indeed, since $\chi(\,\cdot\,;\nu)$ is an analytic non-trivial function it has only isolated zeros. Further, observe that $\phi_{0}(z;\nu)$ and $\chi(z;\nu)$ cannot have a common zero $z>0$ since otherwise, by Lemma~\ref{lem:recur_phi}, one would have $\phi_{n}(z;\nu)=0$ for all $n\in\N$, which is in contradiction with claim 2) of Proposition~\ref{prop:phi_fund_proper}. Next, as we know from Proposition~\ref{prop:point_spec_J_nu}, positive zeros of $\chi(\,\cdot\,;\nu)$ coincide with discrete eigenvalues of~$J_{\nu}$. Their algebraic multiplicities are equal to their orders as poles of the $m$-function of $J_{\nu}$~\cite[Thm.~2.14]{bec_jcam00}. Recall the eigenvalues of $J_{\nu}$ are always simple since the eigenvector is uniquely determined by its first entry and $J_{\nu}$ is self-adjoint. Hence we see that the zeros (if any) of $\chi(\,\cdot\,;\nu)$ in $(0,\infty)$ are simple. This means that the measure $\mu_{\nu}$ is only atomic (or void) in $(-\infty,1/4)$ and the corresponding weight is determined by the residue of the $m$-function at $x_{0}$, i.e., 
 \[
  \mu_{\nu}\!\left(\left\{\frac{1}{4}-x_{0}^{2}\right\}\right)=\frac{\phi_{0}(x_{0};\nu)}{\partial_{x}\chi(x_{0};\nu)},
 \]  
 provided that $\chi(x_{0};\nu)=0$, $x_{0}>0$. 
 
 Next, we show that $m(\,\cdot\,;\nu)$ has a branch cut in $(1/4,\infty)$. In fact, we show the jump of $\Im m(\,\cdot\,;\nu)$ in $(1/4,\infty)$ is an integrable function and hence determines the density of the absolutely continuous part of $\mu_{\nu}$ by means of the Stieltjes--Perron inversion formula~\cite[Thm.~B.2]{tes_00}. Notice that $\lambda(z)\to1/4+x^{2}$, for $x>0$, from the lower half-plane $\Im\lambda<0$ if and only if $z\to\ii x$ from the right half-plane $\Re z>0$. Taking also Proposition~\ref{prop:phi_fund_proper} into account, we compute
 \begin{align*}
  h(x)&:=\lim_{\substack{\lambda\to1/4+x^{2} \\ \Im\lambda<0}} \Im m(\lambda;\nu)= \frac{1}{2\ii}\lim_{\substack{z\to\ii x \\ \Re z>0}}\left[m(\lambda(z);\nu)-m(\lambda(\overline{z});\nu)\right]\\
  &=\frac{1}{2\ii}\left[\frac{\phi_{0}(\ii x;\nu)}{\chi(\ii x;\nu)}-\frac{\phi_{0}(-\ii x;\nu)}{\chi(-\ii x;\nu)}\right]
  =\frac{W(\phi(-\ii x;\nu),\phi(\ii x;\nu))}{2\ii\left|\chi(\ii x;\nu)\right|^{2}}
  =-\frac{\sinh(2\pi x)}{2\pi\left|\chi(\ii x;\nu)\right|^{2}},
 \end{align*}
 for $x>0$. Note that the function $h$ is continuous in $[0,\infty)$. Moreover, a straightforward application of the Stirling formula in~\eqref{eq:chi_simple_hyp_form2} yields
 \[
  \left|\chi(\ii x;\nu)\right|^{2}=\frac{\Gamma^{2}(\nu)\Gamma^{2}(\nu+1)}{8\pi^{3}}x^{2-4\nu}e^{3\pi x}\left[1+O\left(\frac{1}{x}\right)\right], \quad x\to\infty.
 \] 
 Thus, $h\in L^{1}(0,\infty)$. By the Stieltjes--Perron inversion formula, $\mu_{\nu}$ is absolutely continuous in $[1/4,\infty)$ and its density is given by the equation 
 \[
 \frac{\dd\mu_{\nu}}{\dd x}\!\left(\frac{1}{4}+x^{2}\right)=-\frac{1}{\pi} h(x),
 \]
 for $x>0$.
\end{proof}

As an immediate corollary of Proposition~\ref{prop:spectral_meas_J_nu}, we have the spectrum of~$J_{\nu}$ and its parts.

\begin{cor}\label{cor:spectrum_J_nu}
For any $\nu\in\R\setminus(-\N_{0})$, the spectrum of $J_{\nu}$ is simple and decomposes as $\sigma(J_{\nu})=\sigma_{\p}(J_{\nu})\cup\sigma_{\ac}(J_{\nu})$, where 
 \[
 \sigma_{\ac}(J_{\nu})=\bigg[\frac{1}{4},\infty\bigg) 
 \quad\mbox{ and }\quad 
  \sigma_{\p}(J_{\nu})=\left\{\frac{1}{4}-x^{2} \;\bigg|\; \chi(x;\nu)=0, \; x>0\right\}\!.
 \]
 Moreover, $\sigma_{\p}(J_{\nu})$ is finite (possibly empty).
\end{cor}

\begin{proof}
 The only claim to be proven is the finiteness of $\sigma_{\p}(J_{\nu})$. By Proposition~\ref{lem:J_positive}, $J_{\nu}\geq\min(0,\nu)$. Hence $\sigma_{p}(J_{\nu})\subset[\min(0,\nu),1/4)$. This means that the positive zeros of $\chi(\,\cdot\,;\nu)$, which determine eigenvalues of $J_{\nu}$, are located in the compact interval $[0,\sqrt{1/4-\min(0,\nu)}]$. Since $\chi(\,\cdot\,;\nu)$ is a non-trivial analytic function the set of these zeros has to be finite. 
\end{proof}

\subsection{The point spectrum of $J_{\nu}$ in more detail for $\nu>0$}
\label{subsec:point_spectrum_J_nu_nu>0}

As our initial motivation is to determine $\|L_{\nu}\|$, for $\nu>0$, we investigate more closely the point spectrum of $J_{\nu}$. First, we prove an auxiliary result on monotonicity properties of function $\chi(x;\nu)$ in both variables, which will be needed below.

\begin{lem}\label{lem:monotonicity_prop_aux}
 The following properties hold:
 \begin{enumerate}[{\upshape i)}]
  \item The function
  \[
   \nu\mapsto\frac{\Gamma(x+\nu+1/2)}{\Gamma(\nu+1)}\chi(x;\nu)
  \]
  is strictly increasing on $(0,\infty)$, if $x\in[0,1/2)$, and strictly decreasing on $(0,\infty)$, if $x\in(1/2,\infty]$. 
  \item If $x\geq 0, \nu>0$, and $x+\nu\geq 1/2$, then $\chi(x;\nu)>0$. In particular, $\chi(\,\cdot\,;\nu)$ has no positive zero, if $\nu\geq1/2$.
  \item For all $\nu\in(0,1/2]$, the function 
  \[  
  x\mapsto \frac{\Gamma(x+1/2)\Gamma^{2}(x+\nu+1/2)}{x+1/2}\,\chi(x;\nu)
  \]  
  is strictly increasing on $[0,1/2]$.
 \end{enumerate}
\end{lem}

\begin{proof}
 i) According to~\eqref{eq:chi_simple_hyp_form1}, we have
 \[
  \frac{\Gamma(x+\nu+1/2)}{\Gamma(\nu+1)}\chi(x;\nu)=\frac{x+1/2}{\Gamma(2x+1)}\,\pFq{3}{2}{x-1/2,x+1/2,x+3/2}{2x+1,x+\nu+1/2}{1},
  \]
  for all $x\geq0$ and $\nu>0$. Hence it suffices to prove that the function
  \[
   \nu \mapsto \pFq{3}{2}{x-1/2,x+1/2,x+3/2}{2x+1,x+\nu+1/2}{1}
  \]
  has the monotonicity property of claim 1).
  By definition~\eqref{eq:def_3F2}, we have the series representation
  \[
  \pFq{3}{2}{x-1/2,x+1/2,x+3/2}{2x+1,x+\nu+1/2}{1}\!=1+\sum_{k=1}^{\infty}\frac{(x-1/2)_{k}(x+1/2)_{k}(x+3/2)_{k}}{k!\,(2x+1)_{k}}\frac{1}{(x+\nu+1/2)_{k}}.
  \]
  Now, it suffices to note that, for all $k\in\N$ and all $x\geq0$, the function 
  \[
  \nu\mapsto\frac{1}{(x+\nu+1/2)_{k}}
  \]
  is strictly decreasing on $(0,\infty)$ while, for all $k\in\N$, the term 
  \[
  \frac{(x-1/2)_{k}(x+1/2)_{k}(x+3/2)_{k}}{k!\,(2x+1)_{k}}
  \]
  is negative, if $x\in[0,1/2)$, and positive, if $x>1/2$.
  
  ii) By inspection of the expression 
  \[
   \chi(x;\nu)=\frac{(x+1/2)\Gamma(\nu)\Gamma(\nu+1)}{\Gamma(x+1/2)}\sum_{k=0}^{\infty}\frac{(\nu)_{k}^{2}(x+1/2)_{k}}{k!\,\Gamma(x+\nu-1/2+k)\Gamma(x+\nu+3/2+k)},
  \]
  which follows from~\eqref{eq:chi_simple_hyp_form3}, we see that each term is positive provided that  $x\geq0$, $\nu>0$ and $x+\nu>1/2$. In the limiting case $x+\nu=1/2$, the first term of the sum vanishes while the remaining terms remain all positive.
  
  iii) By making use of formula~\eqref{eq:chi_simple_hyp_form2}, we obtain 
  \[
   \frac{\Gamma(x+1/2)\Gamma^{2}(x+\nu+1/2)}{x+1/2}\,\chi(x;\nu)=\Gamma(\nu)\Gamma(\nu+1)\sum_{k=0}^{\infty}\frac{(\nu-1)_{k}(\nu+1)_{k}}{k!}\prod_{j=0}^{k-1}g_{j}(x;\nu),
  \]
  where
  \[
   g_{j}(x;\nu):=\frac{x+j+1/2}{(x+\nu+j+1/2)^{2}}.
  \]
   Since, for all $j\in\N_{0}$, $\nu\in(0,1/2]$, and $x>0$, we have
  \[
   \frac{\partial g_{j}}{\partial x}(x;\nu)=-\frac{x-\nu+j+1/2}{(x+\nu+j+1/2)^{3}}\leq-\frac{x}{(x+\nu+j+1/2)^{3}}<0,
  \]
  functions $g_{j}(\,\cdot\,;\nu)$ are strictly decreasing for all $j\in\N_{0}$. Noting also that $(\nu-1)_{k}<0$ for all $k\in\N$ and $\nu\in(0,1)$, we obtain the claim.
\end{proof}

Now, we are in position to prove that, if $\nu>0$, the point spectrum of $J_{\nu}$ is either empty or a~one-point set.

\begin{thm}\label{thm:zeros_nu0_and_x0} The following claims hold:
\begin{enumerate}[1)]
\item The function
\[
 \nu\mapsto\pFq{3}{2}{-1/2,1/2,3/2}{1,\nu+1/2}{1}
\]
has a unique positive zero $\nu_{0}$ which is located in $(0,1/2)$. 
\item 
We have
\[
 \sigma_{p}(J_{\nu})=\begin{cases}
  \emptyset,& \mbox{ if } \nu\geq\nu_{0},\\
  1/4-x_{0}^{2}(\nu),& \mbox{ if } 0<\nu<\nu_{0},
  \end{cases}
\]
where $x_{0}(\nu)$ is the unique zero of the function
\[
x\mapsto\pFq{3}{2}{x-1/2,x+1/2,x+3/2}{2x+1,x+\nu+1/2}{1}
\]
located in $(0,1/2)$.
\item Function $x_{0}:(0,\nu_{0})\to(0,1/2):\nu\mapsto x_{0}(\nu)$ is real analytic and strictly decreasing.
\end{enumerate}
\end{thm}

\begin{rem}
 Numerically, one has $\nu_{0}\approx0.349086$.
\end{rem}

\begin{proof}[Proof of Theorem~\ref{thm:zeros_nu0_and_x0}]
1) It follows from~\eqref{eq:chi_simple_hyp_form1} that 
\[
 g(\nu):=\pFq{3}{2}{-1/2,1/2,3/2}{1,\nu+1/2}{1}=2\frac{\Gamma(\nu+1/2)}{\Gamma(\nu+1)}\chi(0;\nu), 
\]
for $\nu>0$. According to claim~(i) of Lemma~\ref{lem:monotonicity_prop_aux}, $g$ is strictly increasing on $(0,\infty)$. Further, one computes readily from~\eqref{eq:chi_simple_hyp_form3} that 
 \[
 \lim_{\nu\to0}\nu\,\chi(0;\nu)=\frac{1}{\Gamma^{2}(1/2)\Gamma(-1/2)}=-\frac{1}{2\pi^{3/2}}<0,
 \]
 which means that $\lim_{\nu\to0+}g(\nu)=-\infty$. By using formula~\eqref{eq:chi_simple_hyp_form3} again, one gets
 \[
  g\!\left(\frac{1}{2}\right)=\regpFq{3}{2}{1/2,1/2,1/2}{0,2}{1}=\sum_{k=1}^{\infty}\frac{(1/2)_{k}^{3}}{(k-1)!\,k!\,(k+1)!}>0.
 \]
 These properties imply that $g$ has a unique positive zero $\nu_{0}$, which is located in $(0,1/2)$.
 
 2) For $x,\nu>0$, we temporarily denote
 \[
 r(x;\nu):=\pFq{3}{2}{x-1/2,x+1/2,x+3/2}{2x+1,x+\nu+1/2}{1}
 \]
 and
 \[
 s(x;\nu):=\frac{\Gamma(x+1/2)\Gamma^{2}(x+\nu+1/2)}{x+1/2}\,\chi(x;\nu).
 \]
 In view of~\eqref{eq:chi_simple_hyp_form1}, we have
 \[
 s(x;\nu)=\frac{\Gamma(\nu+1)\Gamma(x+1/2)\Gamma(x+\nu+1/2)}{\Gamma(2x+1)}r(x;\nu).
 \]
 Thus, for $\nu>0$, the set of positive zeros of functions $r(\,\cdot\,;\nu)$, $s(\,\cdot\,;\nu)$, and $\chi(\,\cdot\,;\nu)$ coincide.
 It follows from claim~(ii) of Lemma~\ref{lem:monotonicity_prop_aux} that, if $\nu\geq1/2$, $r(\,\cdot\,;\nu)$ has no positive zero.

Next, suppose $\nu\in(0,1/2)$. According to claim~(iii) of Lemma~\ref{lem:monotonicity_prop_aux}, function $s(\,\cdot\,;\nu)$ is strictly increasing on $[0,1/2]$. Further, observe that 
\[
 s(0;\nu)=\Gamma(1/2)\Gamma(\nu+1)\Gamma(\nu+1/2)g(\nu)
 \quad\mbox{ and }\quad 
 s\left(\frac{1}{2};\nu\right)=\Gamma^{2}(\nu+1).
\]
Consequently, $s(\,\cdot\,;\nu)$ has a unique zero $x_{0}(\nu)\in(0,1/2)$ if and only if $g(\nu)<0$, which is further equivalent to the condition $\nu<\nu_{0}$. Otherwise, if $\nu\geq\nu_{0}$, $s(x;\nu)>0$ for all $x\in(0,1/2]$. This proves the assertion concerning zeros of $r(\,\cdot\,;\nu)$ from the statement. The claim on the point spectrum of $J_{\nu}$ now follows from Proposition~\ref{prop:point_spec_J_nu}.

3) The zero $x_{0}$ depends analytically on $\nu$ by the analytic Implicit Function Theorem. Next, by differentiating the implicit equation
\[
 s(x_{0}(\nu);\nu)=0
\]
with respect to $\nu$, we get
\begin{equation}
 \frac{\partial s}{\partial x}(x_{0}(\nu);\nu)\,x_{0}'(\nu)+\frac{\partial s}{\partial \nu}(x_{0}(\nu);\nu)=0
\label{eq:impl_func_der_x_0}
\end{equation}
for all $\nu\in(0,\nu_{0})$. As we know from claim~(iii) of Lemma~\ref{lem:monotonicity_prop_aux},
\[
\frac{\partial s}{\partial x}(x_{0}(\nu);\nu)>0.
\]
Similarly, by claim~(i) of Lemma~\ref{lem:monotonicity_prop_aux}, we have
\[
 \frac{\partial s}{\partial \nu}(x_{0}(\nu);\nu)=\frac{\Gamma(\nu+1)\Gamma(x+1/2)\Gamma(x+\nu+1/2)}{x+1/2}\,\frac{\partial}{\partial\nu}\bigg|_{x=x_{0}(\nu)}\!\frac{\Gamma(x+\nu+1/2)}{\Gamma(\nu+1)}\chi(x;\nu)>0.
\]
Thus, it follows from~\eqref{eq:impl_func_der_x_0} that $x'(\nu)<0$ for all $\nu\in(0,\nu_{0})$. In other words, $x_{0}$ is strictly decreasing on $(0,\nu_{0})$.

To conclude that $x_{0}$ maps $(0,\nu_{0})$ onto $(0,1/2)$, it suffices to check that the limit values
\[
 \lim_{\nu\to 0+}x_{0}(\nu)=\frac{1}{2} \quad\mbox{ and }\quad
 \lim_{\nu\to\nu_{0}-}x_{0}(\nu)=0.
\]
The latter limit relation follows from the proof of claim 2), where we have observed that $s(0;\nu_{0})=0$ and $s(x;\nu_{0})>0$ for all $x>0$. Next, the already proven claim 2) together with the min-max principle implies the estimate
\begin{equation}
 \frac{1}{4}-x_{0}^{2}(\nu)=\inf_{\substack{\phi\in\Dom J_{\nu} \\ \|\phi\|=1}}\langle \phi,J_{\nu}\phi\rangle\leq\langle e_{0},J_{\nu}e_{0}\rangle=\nu(\nu+1),
\label{eq:min-max_x0_inproof}
\end{equation}
for $\nu\in(0,\nu_{0})$. Thus, we have
\[
 \sqrt{\frac{1}{4}-\nu(\nu+1)}\leq x_{0}(\nu)\leq\frac{1}{2},
\]
for $\nu\in(0,\nu_{0})$, which implies that $x_{0}(\nu)\to1/2$, as $\nu\to0+$.
\end{proof}

Next, we show that the bound~\eqref{eq:min-max_x0_inproof} on the bottom of the spectrum of $J_{\nu}$ can be improved when a different choice of the test sequence for the min-max principle is made. The result actually holds true for all $\nu\in\R\setminus(-\N_{0})$.

\begin{prop}\label{prop:bounds_inf_spec_J_nu}
 For $\nu\in\R\setminus(-\N_{0})$, one has
 \[
  \min(0,\nu)\leq\inf\sigma(J_{\nu})\leq\min\left(\frac{1}{4},\frac{1}{\nu\,\psi'(\nu)}\right),
 \]
 where $\psi=\Gamma/\Gamma'$ is the Digamma function.
\end{prop}

\begin{proof}
 The lower bound $\min(0,\nu)\leq\inf\sigma(J_{\nu})$ follows immediately from Lemma~\ref{lem:J_positive}. Further, since $\sigma_{\ac}(J_{\nu})=[1/4,\infty)$ by Corollary~\ref{cor:spectrum_J_nu}, $\inf\sigma(J_{\nu})\leq1/4$. Thus, it suffices to prove the upper bound $\inf\sigma(J_{\nu})\leq1/\left(\nu\,\psi'(\nu)\right)$.
 
By the min-max principle,
 \[
  \inf\sigma(J_{\nu})=\inf_{0\neq\phi\in\Dom J_{\nu}}\frac{\langle\phi,J_{\nu}\phi\rangle}{\|\phi\|^{2}}.
 \]
 We choose the test sequence $\phi^{*}$ with entries
 \[
 \phi_{n}^{*}:=\frac{1}{n+\nu}, \quad n\in\N_{0}.
 \]
 Then $J_{\nu}\phi^{*}=e_{0}$, hence $\langle\phi^{*},J_{\nu}\phi^{*}\rangle=1/\nu$. Further, we have
 \[
  \|\phi^{*}\|^{2}=\sum_{n=0}^{\infty}\frac{1}{(n+\nu)^{2}}=\psi'(\nu),
 \]
 where we used the well-known identity~\cite[Eq.~5.15.1]{dlmf}. The desired estimate now follows from the min-max principle.
\end{proof}

As we know from claim 3) of Theorem~\ref{thm:zeros_nu0_and_x0}, $x_{0}(\nu)\to1/2$, as $\nu\to0+$, which is reasonable since $0\in\sigma_{\p}(J_{0})$. However, a more detailed information on the asymptotic behavior of $x_{0}(\nu)$, for $\nu\to0+$, can be computed.

\begin{prop}\label{prop:asympt_x_0}
For $\nu\to0+$, we have the asymptotic expansion
\begin{equation}
 x_{0}(\nu)=\frac{1}{2}-\nu-\nu^{2}-\left(2-\frac{\pi^{2}}{6}\right)\nu^{3}-\left(5-\frac{\pi^{2}}{3}-\zeta(3)\right)\nu^{4}+O(\nu^{5}),
 \label{eq:asympt_x_0}
\end{equation}
where $x_{0}(\nu)$ is as in Theorem~\ref{thm:zeros_nu0_and_x0} and $\zeta(3)$ is Ap{\' e}ry's constant.
\end{prop}

\begin{proof}
 We derive the asymptotic expansion of $x_{0}(\nu)$ up to the quadratic term $\nu^{2}$ here. The computation of the other two coefficients by $\nu^{3}$ and $\nu^{4}$ follows the same procedure but is lengthy, therefore is postponed to Appendix~\ref{subsec:a.2}.
 
In view of formula~\eqref{eq:chi_hyper_form}, $x_{0}(\nu)$ is the zero of function
\[
\left(x+\nu-\frac{1}{2}\right)^{\!2}\!\!\pFq{3}{2}{x-1/2,\nu-1,\nu-1}{\nu+x-1/2,\nu+x-1/2}{1}-\nu(\nu-1)\pFq{3}{2}{x-1/2,\nu,\nu}{\nu+x+1/2,\nu+x+1/2}{1},
\]
located in a left neighborhood of the point $1/2$ for $\nu>0$ small. We denote $y:=x-1/2$ and 
\begin{equation}
 h(y;\nu):=\left(y+\nu\right)^{2}\pFq{3}{2}{y,\nu-1,\nu-1}{\nu+y,\nu+y}{1}-\nu(\nu-1)\pFq{3}{2}{y,\nu,\nu}{\nu+y+1,\nu+y+1}{1}.
 \label{eq:def_h_inproof}
\end{equation}
Note that $h$ is analytic in both variables in a neighborhood of zero. Expression~\eqref{eq:def_h_inproof} is in a~suitable form for an asymtotic expansion for $\nu$ small. Using definition~\eqref{eq:def_3F2} and separating the first two terms of the first hypergeometric series from~\eqref{eq:def_h_inproof} and the very first term from the second hypergeometric series in~\eqref{eq:def_h_inproof}, we arrive at the expression
\begin{align}
h(y;\nu)&=(y+\nu)^{2}+y(\nu-1)^{2}+y(y+1)\nu^{2}U(y;\nu)-\nu(\nu-1)-y\nu^{3}V(y;\nu)\nonumber\\
&=y(y+1)+\nu+y\nu^{2}+y(y+1)\nu^{2}U(y;\nu)-y\nu^{3}V(y;\nu),
\label{eq:h_formula_inproof}
\end{align}
where
\begin{equation}
U(y;\nu):=\frac{(\nu-1)^{2}}{(y+\nu+1)^{2}}\sum_{k=0}^{\infty}\frac{(y+2)_{k}(\nu+1)_{k}^{2}}{(y+\nu+2)_{k}^{2}(k+2)!}
\label{eq:def_U}
\end{equation}
and
\begin{equation}
V(y;\nu):=\frac{\nu-1}{(y+\nu+1)^{2}}\sum_{k=0}^{\infty}\frac{(y+1)_{k}(\nu+1)_{k}^{2}}{(y+\nu+2)_{k}^{2}(k+1)!}.
\label{eq:def_V}
\end{equation}
It is easy to see that both $U(y;\nu)$ and $V(y;\nu)$ are uniformly bounded in $y$ and $\nu$ from a~neighborhood of zero. Writing the analytic function 
\[
 y_{0}(\nu):=x_{0}(\nu)-\frac{1}{2}=y_{1}\nu+O(\nu^{2}),
\]
one infers from~\eqref{eq:h_formula_inproof} that
\[
 0=h(y_{0}(\nu);\nu)=(y_{1}+1)\nu+O(\nu^{2}),
\]
as $\nu\to0$. Hence $y_{1}=-1$, i.e.,
\[
y_{0}(\nu)=-\nu+O(\nu^{2}), \quad \nu\to0.
\]

Repeating the same approach for $y_{0}(\nu)$ expanded up to the second term 
\[
 y_{0}(\nu)=-\nu+y_{2}\nu^{2}+O(\nu^{3}),
\]
one obtains
\[
 0=h(y_{0}(\nu);\nu)=(y_{2}+1)\nu^{2}+O(\nu^{3}), \quad \nu\to0,
\]
which yields $y_{2}=-1$. Hence 
\[
y_{0}(\nu)=-\nu-\nu^{2}+O(\nu^{3}), \quad \nu\to0,
\]
getting the first three terms from the expansion to be proved. Two more terms are computed in Appendix~\ref{subsec:a.2}.
\end{proof}

\subsection{Applications for orthogonal polynomials}
\label{subsec:ortohogonal_polynomials}

A straightforward application of Proposition~\ref{prop:spectral_meas_J_nu} yields the orthogonality measure for orthogonal polynomials $p_{n}(x;\nu)$ corresponding to $J_{\nu}$, i.e., determined recursively by
\[
 p_{n+1}(x;\nu)=\left(x-2(n+\nu)^{2}\right)p_{n}(x;\nu)-(n+\nu-1)^{2}(n+\nu)^{2}p_{n-1}(x;\nu), \quad n\in\N,
\]
and $p_{0}(x;\nu)=1$, $p_{1}(x;\nu)=x-\nu(\nu+1)$. Recall the respective orthonormal polynomials $\{P_{n}(\,\cdot\,;\nu)\}_{n=0}^{\infty}$ are determined by the generalized eigenvalue equation $J_{\nu}P(x;\nu)=xP(x;\nu)$, where $P(x;\nu)=(P_{0}(x;\nu),P_{1}(x;\nu),\dots)^{T}$, normalized by $P_{0}(x;\nu)=1$.
Polynomials $p_{n}(\,\cdot\,;\nu)$ and $P_{n}(\,\cdot\,;\nu)$ are simply related by the formula
\begin{equation}
 P_{n}(x;\nu):=\frac{(-1)^{n}}{(\nu)_{n}(\nu+1)_{n}}p_{n}(x;\nu), \quad n\in\N_{0}.
\label{eq:P_rel_p}
\end{equation}
Then the identity $e_{n}=P_{n}(J_{\nu};\nu)e_{0}$, which holds for all $n\in\N_{0}$ and can be easily verified by induction, together with the Spectral Theorem implies
\[
 \delta_{m,n}=\langle e_{m},e_{n} \rangle=\langle P_{m}(J_{\nu};\nu)e_{0},P_{n}(J_{\nu};\nu)e_{0}\rangle=\int_{\R}P_{n}(x;\nu)P_{m}(x;\nu)\, \dd\langle e_{0},E_{J_{\nu}}e_{0}\rangle,
\]
for all $m,n\in\N_{0}$. This well known fact shows that the orthogonality measure for $P_{n}(\,\cdot\,;\nu)$ coincides with the spectral measure $\mu_{\nu}=\langle e_{0},E_{J_{\nu}}e_{0}\rangle$ described in Proposition~\ref{prop:spectral_meas_J_nu}. Taking also into account formula~\eqref{eq:P_rel_p}, we obtain the orthogonality relation for polynomials $p_{n}(\,\cdot\,;\nu)$.

\begin{thm}
 For $\nu\in\R\setminus(-\N_{0})$ and $m,n\in\N_{0}$, we have the orthogonality relation
 \begin{align*}
 &\frac{1}{\pi^{2}}\int_{0}^{\infty}p_{m}\!\left(\frac{1}{4}+x^{2};\nu\right)p_{n}\!\left(\frac{1}{4}+x^{2};\nu\right)\frac{x\sinh(2\pi x)}{|\chi(\ii x;\nu)|^{2}}\,\dd x\\
 &\hskip70pt +\sum_{\substack{x_{0}>0 \\ \chi(x_{0};\nu)=0}}\frac{\phi_{0}(x_{0};\nu)}{\partial_{x}\chi(x_{0};\nu)}\,p_{m}\!\left(\frac{1}{4}-x_{0}^{2};\nu\right)p_{n}\!\left(\frac{1}{4}-x_{0}^{2};\nu\right)=(\nu)_{n}^{2}(\nu+1)_{n}^{2}\,\delta_{m,n}.
 \end{align*}
\end{thm}

If we additionally suppose $\nu>0$, we have more detailed description of the discrete part of the orthogonality measure due to Theorem~\ref{thm:zeros_nu0_and_x0}.

\begin{thm}\label{thm:orthogonality_nu_positive}
 Let $\nu>0$ and $\nu_{0}$, $x_{0}(\nu)$ as defined in Theorem~\ref{thm:zeros_nu0_and_x0}. If $\nu\geq\nu_{0}$, we have
\[
 \frac{1}{\pi^{2}}\int_{0}^{\infty}p_{m}\!\left(\frac{1}{4}+x^{2};\nu\right)p_{n}\!\left(\frac{1}{4}+x^{2};\nu\right)\frac{x\sinh(2\pi x)}{|\chi(\ii x;\nu)|^{2}}\,\dd x=(\nu)_{n}^{2}(\nu+1)_{n}^{2}\,\delta_{m,n},
\]
while if $\nu<\nu_{0}$, we have
 \begin{align*}
 &\frac{1}{\pi^{2}}\int_{0}^{\infty}p_{m}\!\left(\frac{1}{4}+x^{2};\nu\right)p_{n}\!\left(\frac{1}{4}+x^{2};\nu\right)\frac{x\sinh(2\pi x)}{|\chi(\ii x;\nu)|^{2}}\,\dd x\\
 &\hskip70pt +\frac{\phi_{0}(x_{0}(\nu);\nu)}{\partial_{x}\chi(x_{0}(\nu);\nu)}\,p_{m}\!\left(\frac{1}{4}-x_{0}^{2}(\nu);\nu\right)p_{n}\!\left(\frac{1}{4}-x_{0}^{2}(\nu);\nu\right)=(\nu)_{n}^{2}(\nu+1)_{n}^{2}\,\delta_{m,n},
 \end{align*}
 for all $m,n\in\N_{0}$.
\end{thm}

\begin{rem}
 With the aid of formula~\eqref{eq:chi_simple_hyp_form2} and identity~\cite[Eq.~5.4.4]{dlmf}
 \[
  \left|\Gamma\left(\frac{1}{2}+\ii x\right)\right|^{2}=\frac{\pi}{\cosh(\pi x)}, \quad x\in\R,
 \]
 one verifies that 
 \[
  \frac{1}{\pi^{2}}\frac{x\sinh(2\pi x)}{|\chi(\ii x;1)|^{2}}=2\pi x\left(\frac{1}{4}+x^{2}\right)\frac{\sinh(\pi x)}{\cosh^{2}(\pi x)}.
 \]
 It means that, for $\nu=1$, Theorem~\ref{thm:orthogonality_nu_positive} simplifies to the well-known orthogonality~\eqref{eq:orthogonality_nu=1} (recall that $\nu_{0}\in(0,1/2)$).
\end{rem}

\begin{rem}
 In view of Proposition~\ref{prop:phi_fund_proper}, orthonormal polynomials $P_{n}(\,\cdot\,;\nu)$ can be expressed in the form 
 \[
  P_{n}\!\left(\frac{1}{4}-z^{2};\nu\right)=\frac{\chi(-z;\nu)\phi_{n}(z;\nu)-\chi(z;\nu)\phi_{n}(-z;\nu)}{\chi(-z;\nu)\phi_{0}(z;\nu)-\chi(z;\nu)\phi_{0}(-z;\nu)}
 \]
 since then equations $J_{\nu}P(z;\nu)=z P(z;\nu)$ and $P_{0}(z;\nu)=1$ are fulfilled. Noticing also that the denominator coincides with $W(\phi(-z;\nu),\phi(z;\nu))$, we may rewrite the last formula as
 \begin{equation}
 P_{n}\!\left(\frac{1}{4}-z^{2};\nu\right)=\frac{\pi}{\sin(2\pi z)}\left[\chi(z;\nu)\phi_{n}(-z;\nu)-\chi(-z;\nu)\phi_{n}(z;\nu)\right],
 \label{eq:P_rel_chi_phi}
 \end{equation}
 for $z\in\C\setminus(\Z/2)$. However, since polynomial in $z$, formula~\eqref{eq:P_rel_chi_phi} extends to all $z\in\C$, if the right-hand side is interpreted as the respective limit value. Then one may combine relations from~\eqref{eq:def_phi} and \eqref{eq:def_phi_extended}--\eqref{eq:chi_simple_hyp_form3}
in order to express polynomials $P_{n}(\;\cdot\;,\nu)$ in terms of hypergeometric functions. Such formulas can be useful when deriving various asymptotic formulas for the orthogonal polynomials.
\end{rem}

\subsection{Spectral properties of the Hilbert $L$-operator}
\label{subsec:spectrum_L_nu}

Since $L_{\nu}=J_{\nu}^{-1}$, for $\nu\in\R\setminus(-\N_{0})$, spectral properties of $L_{\nu}$ can be readily deduced from previous spectral analysis of $J_{\nu}$. In this subsection, spectral properties of the Hilbert $L$-operator $L_{\nu}$ are summarized. Recall function $\chi(\,\cdot\,;\nu)$ is defined in claim 1) of Proposition~\ref{prop:phi_fund_proper} and can be expressed in terms of the hypergeometric functions by one of formulas~\eqref{eq:chi_hyper_form}--\eqref{eq:chi_simple_hyp_form3} depending on restrictions of variables.

\begin{thm}[Spectrum of $L_{\nu}$ for general $\nu$]\label{thm:spectrum_L_nu}
 For all $\nu\in\R\setminus(-\N_{0})$, the spectrum of $L_{\nu}$ is simple and $\sigma(L_{\nu})=\sigma_{\ac}(L_{\nu})\cup \sigma_{\p}(L_{\nu})$, where
 \[
  \sigma_{\ac}(L_{\nu})=[0,4]
 \]
 and
 \[
  \sigma_{\p}(L_{\nu})=\left\{\frac{4}{1-4x^{2}}\;\bigg|\;\chi(x;\nu)=0, \, x>0\right\}.
 \]
 Moreover, $\sigma_{p}(L_{\nu})$ is finite (possibly empty).
\end{thm}

\begin{proof}
 It follows from the spectral properties of $J_{\nu}$ given in Corollary~\ref{cor:spectrum_J_nu}.
\end{proof}

\begin{thm}[Point spectrum of $L_{\nu}$ for $\nu>0$]\label{thm:spectrum_L_nu_nu>0}
 Let $\nu>0$ and $\nu_{0}$, $x_{0}(\nu)$ the roots defined by Theorem~\ref{thm:zeros_nu0_and_x0}.
 \begin{enumerate}[1)]
 \item If $\nu\geq\nu_{0}$, $\sigma_{\p}(L_{\nu})=\emptyset$, while if $\nu<\nu_{0}$, $\sigma_{\p}(L_{\nu})$ is the one-point set containing
 \[
  \|L_{\nu}\|=\frac{4}{1-4x_{0}^{2}(\nu)}.
 \]
 \item Function $\|L_{\nu}\|:(0,\nu_{0})\to(4,\infty)$ is real analytic and strictly decreasing.
 \item We have the lower bound
 \[
  \|L_{\nu}\|\geq\max\left(4,\nu\psi'(\nu)\right),
 \]
 where $\psi=\Gamma'/\Gamma$ is the Digamma function.
 \item For $\nu\to0+$, we have the asymptotic expansion
 \[
  \|L_{\nu}\|=\frac{1}{\nu}+\frac{\pi^{2}}{6}\nu+\zeta(3)\,\nu^{2}+O\!\left(\nu^{3}\right),
 \]
 where $\zeta(3)$ is Ap{\' e}ry's constant.
 \end{enumerate}
\end{thm}

\begin{proof}
 Claim 1) follows from claim 2) of Theorem~\ref{thm:zeros_nu0_and_x0} and the fact that 
 \[
  \|L_{\nu}\|=\sup\sigma(L_{\nu})=\frac{1}{\inf\sigma(J_{\nu})}, 
 \]
 for $\nu>0$. Claim 2) is a consequence of claim 3) of Theorem~\ref{thm:zeros_nu0_and_x0}. Claim 3) follows from Proposition~\ref{prop:bounds_inf_spec_J_nu} and the asymptotic expansion from claim 4) can be readily computed using Proposition~\ref{prop:asympt_x_0}.
\end{proof}

If $\nu<0$ and $-\nu\notin\N$, $L_{\nu}$ is no more positive-semidefinite, which follows from Lemma~\ref{lem:J_positive}. Hence, there exists a negative spectral point of $L_{\nu}$ that must be an eigenvalue according to Theorem~\ref{thm:spectrum_L_nu}. In other words, for $\nu<0$, function $\chi(\,\cdot\,;\nu)$ has a zero greater than $1/2$. On the other hand, it is not clear whether there is also an eigenvalue of $L_{\nu}$ greater than $4$. Actually, the situation seems to be quite delicate since, as indicated by numerical experiments, there is exactly one negative eigenvalue below $[0,4]$ and none or exactly one eigenvalue above $[0,4]$ depending on the value of $\nu<0$. We formulate this claim as the following conjecture. A~more detailed numerical demonstration of the phenomenon is given in Section~\ref{sec:numerics}.

\begin{conjecture}\label{conj}
 Suppose $\nu<0$ and $-\nu\notin\N$. Then $\sigma(L_{\nu})$ consists of exactly one negative eigenvalue and none or exactly one eigenvalue of $L_{\nu}$ greater than $4$. More precisely, there are numbers $-2<\nu_{3}<\nu_{2}<-1<\nu_{1}<0$ such that 
 \[
 \sigma_{\p}(L_{\nu})=\begin{cases}
  \{\lambda_{-}(\nu)\},& \quad \mbox{ for }\; \nu\in(\nu_{3},\nu_{2})\cup(\nu_{1},0),\\
  \{\lambda_{-}(\nu),\lambda_{+}(\nu)\},& \quad \mbox{ otherwise,}
 \end{cases}
 \]
 where $\lambda_{-}(\nu)<0$ and $\lambda_{+}(\nu)>4$.
\end{conjecture}

\section{Illustrative and comparison plots}\label{sec:numerics}

Below, we provide six illustrative and comparison plots concerning spectral properties of $J_{\nu}$ and $L_{\nu}$. First, a comparison of $\inf\sigma(J_{\nu})$ and the upper bound from Proposition~\ref{prop:bounds_inf_spec_J_nu}, for $\nu>0$, is given in Figure~\ref{fig:x0bound}.

Second, in Figure~\ref{fig:rootsnu}, we plot the graph of the function
\begin{equation}
 \nu\mapsto\regpFq{3}{2}{1/2,\nu-1,\nu+1}{\nu+1/2,\nu+1/2}{1}.
\label{eq:3F2_func_rootnu}
\end{equation}
Zeros of~\eqref{eq:3F2_func_rootnu} coincide with zeros of $\chi(0;\nu)$ since the two functions differ by a~non-vanishing factor, see~\eqref{eq:chi_simple_hyp_form2}. As it follows from Theorem~\ref{thm:zeros_nu0_and_x0}, function~\eqref{eq:3F2_func_rootnu} has exactly one positive zero, the number $\nu_{0}$. Taking also negative values of $\nu$ into account, there is numerical evidence that function~\eqref{eq:3F2_func_rootnu} possesses exactly 3 more negative zeros $\nu_{3}<\nu_{2}<\nu_{1}<0$. Numerical values of the zeros are approximately
\[
 \nu_{3}\approx -1.33742, \qquad \nu_{2}\approx -1.1426, \qquad \nu_{1}\approx -0.43215, \qquad \nu_{0}\approx0.34909.
\]

Figure~\ref{fig:zerosx} displays positive zeros of $\chi(\,\cdot\,;\nu)$ as functions of~$\nu$. Recall that $\chi(\,\cdot\,;\nu)$ has no positive zero, if $\nu\geq\nu_{0}$, has exactly one zeros in $(0,1/2)$, if $\nu\in(0,\nu_{0})$, and at least one zero greater than $1/2$, if $\nu<0$. Based on numerical experiments, we may be more precise when $\nu<0$. It seems that, for $\nu<0$, $\chi(\,\cdot\,;\nu)$ has exactly one zero, denoted by $x_{1}(\nu)$, which greater than $1/2$, and possibly one zero $x_{0}(\nu)$ located in $(0,1/2)$. As shown in Figure~\ref{fig:zerosx}, $\chi(\,\cdot\,;\nu)$ has exactly one zero $x_{1}(\nu)$ (greater than $1/2$), when $\nu\in(\nu_{3},\nu_{2})\cup(\nu_{1},0)$. For the remaining values of $\nu<0$, $\nu\notin-\N$, both positive zeros $x_{0}(\nu)$ and $x_{1}(\nu)$ are present.

The observations on positive zeros of $\chi(\,\cdot\,;\nu)$, for $\nu<\nu_{0}$, have immediate consequences on eigenvalues of $J_{\nu}$ due to the characterization of $\sigma_{\p}(J_{\nu})$ from Proposition~\ref{prop:point_spec_J_nu}. This is illustrated by Figure~\ref{fig:evlsJ}, where the positive eigenvalue of $J_{\nu}$ is denoted by $\mu_{+}(\nu)\in(0,1/4)$ and the negative eigenvalue by $\mu_{-}(\nu)$. The corresponding consequences on $\sigma_{\p}(L_{\nu})$ are illustrated by Figure~\ref{fig:evlsL}, where the notation of Conjecture~\ref{conj} is used. Finally, the graph of function $\nu\mapsto\|L_{\nu}\|$, computed as $\|L_{\nu}\|=\max(-\lambda_{-}(\nu),\lambda_{+}(\nu),4)$, is plotted in Figure~\ref{fig:normL}.

\newpage

\[
\,
\]

\begin{figure}[htb!]
	\includegraphics[width=0.99\textwidth]{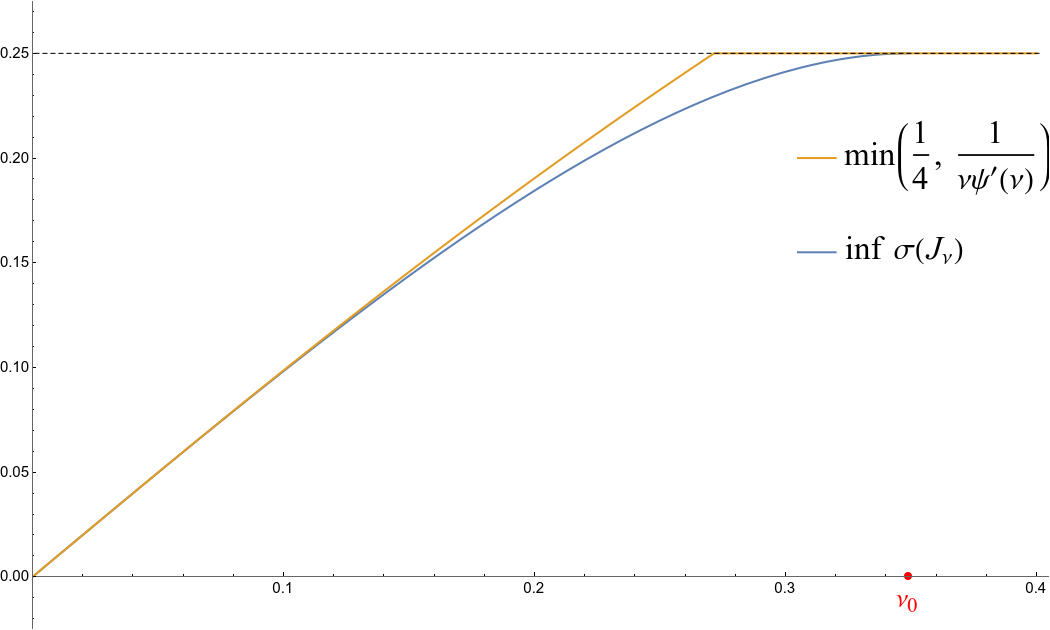}
	\vskip-8pt
	\caption{A comparison of $\inf\sigma(J_{\nu})$ and the upper bound from Proposition~\ref{prop:bounds_inf_spec_J_nu}, for $0<\nu<0.4$.}
	\label{fig:x0bound}
\end{figure}

\[
\,
\]

\begin{figure}[htb!]
	\includegraphics[width=0.99\textwidth]{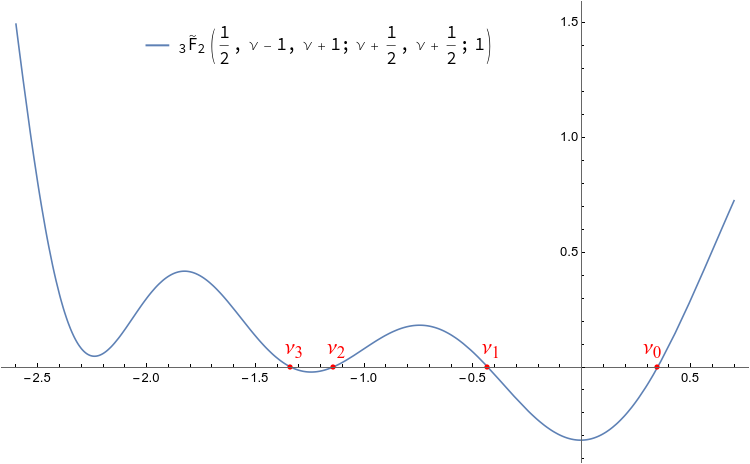}
	\vskip-8pt
	\caption{Function~\eqref{eq:3F2_func_rootnu} and its zeros.}
	\label{fig:rootsnu}
\end{figure}

\newpage

\[
\,
\]

\begin{figure}[htb!]
	\includegraphics[width=0.99\textwidth]{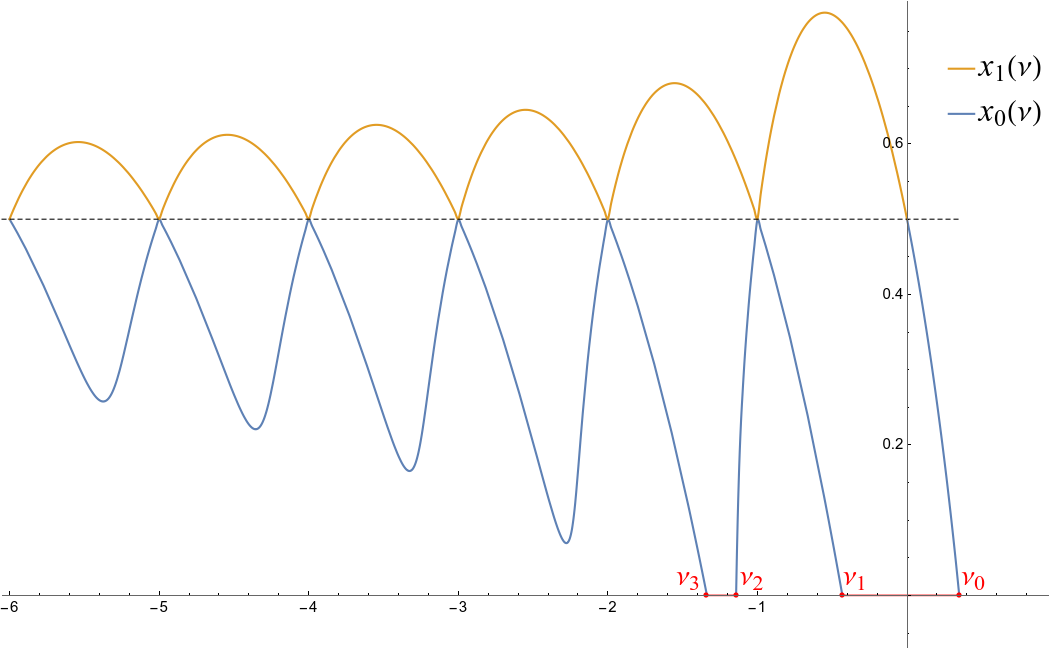}
    \vskip-8pt
	\caption{Positive zeros of function $\chi(\,\cdot\,;\nu)$.}
	\label{fig:zerosx}
\end{figure}

\[
\,
\]

\begin{figure}[htb!]
	\includegraphics[width=0.99\textwidth]{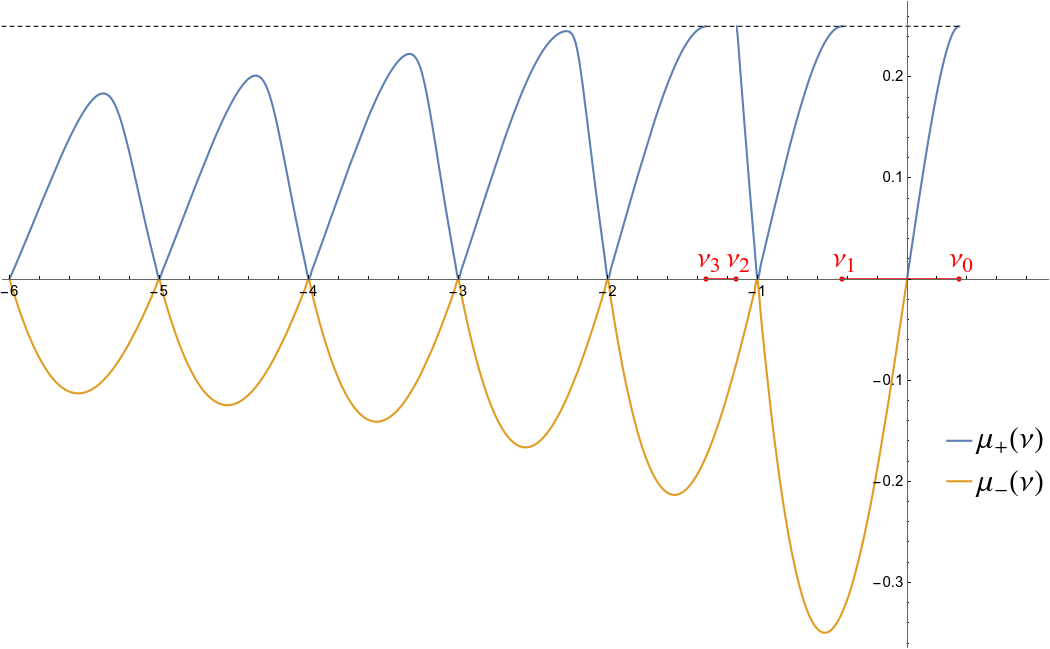}
	\vskip-8pt
	\caption{Eigenvalues of $J_{\nu}$ as functions of $\nu$.}
	\label{fig:evlsJ}
\end{figure}

\newpage

\[
\,
\]

\begin{figure}[htb!]
	\includegraphics[width=0.99\textwidth]{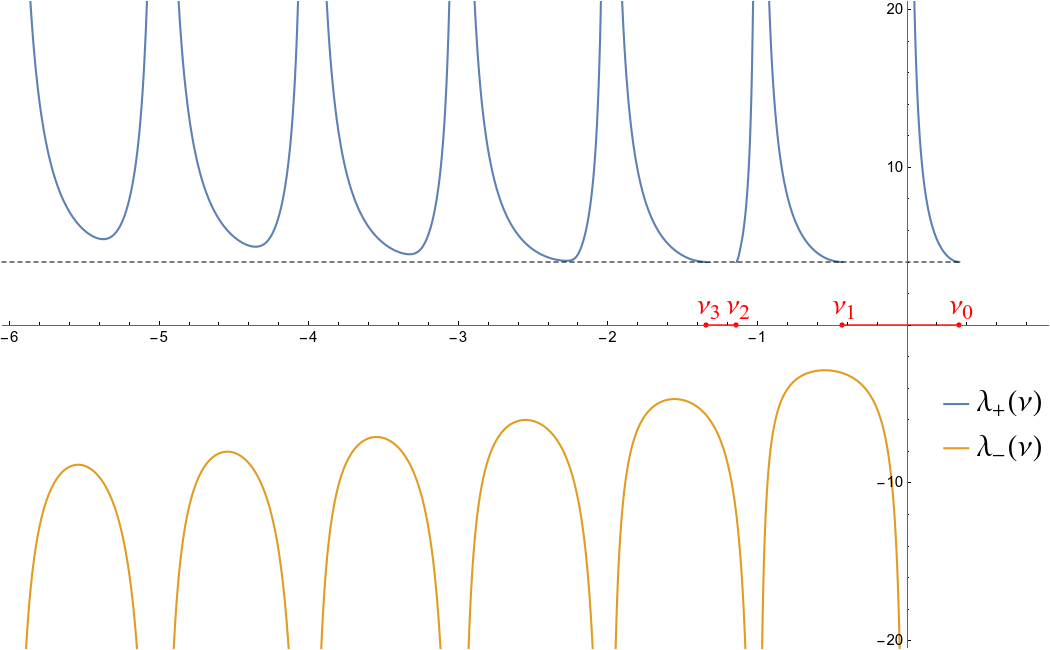}
	\caption{Eigenvalues of $L_{\nu}$ as function of $\nu$.}
	\label{fig:evlsL}
\end{figure}
\vskip12pt

\[
\,
\]

\begin{figure}[htb!]
	\includegraphics[width=0.9\textwidth]{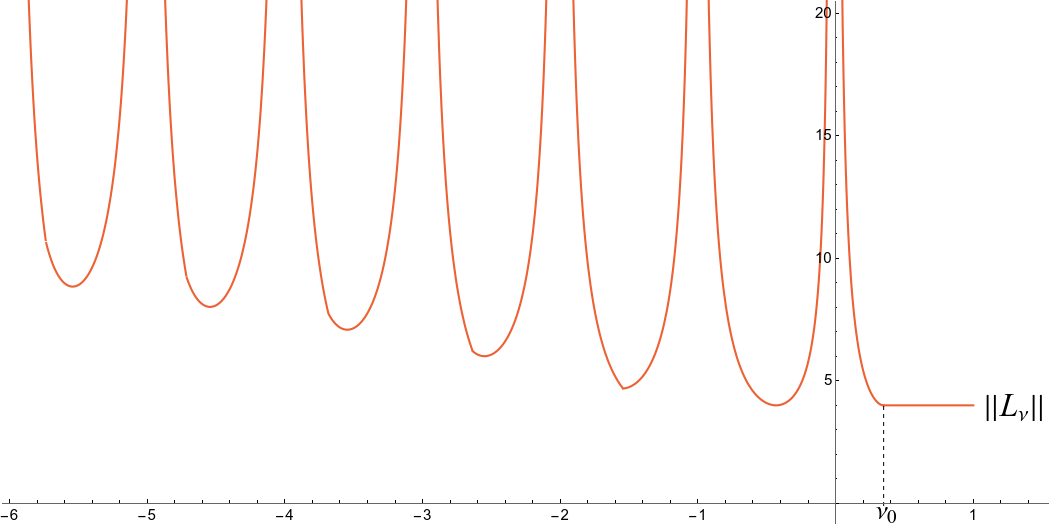}
	\caption{The norm of $L_{\nu}$ as function of $\nu$.}
	\label{fig:normL}
\end{figure}

\section*{Acknowledgement}
The author wishes to acknowledge gratefully partial support from grant No.~20-17749X of the Czech Science Foundation.

\setcounter{section}{1}
\renewcommand{\thesection}{\Alph{section}}
\setcounter{equation}{0} \renewcommand{\theequation}{\Alph{section}.\arabic{equation}}

\section*{Appendix}

\subsection{The method of successive approximations}\label{subsec:a.1}
  
The full description of the method of successive approximations is explained in~\cite{li-wong_jcam92}. We recall only a very particular case, which applies to difference equation~\eqref{eq:diff_eq_z=0} of our interest. 

We consider the difference equation in the form
\begin{equation}
 \phi_{n+2}+a_{n}\phi_{n+1}+b_{n}\phi_{n}=0,
\label{eq:diff_eq_gen_wong_li}
\end{equation}
where coefficients $a_{n}$ and $b_{n}$ have the asymptotic behavior
\[
 a_{n}=\alpha_{0}+\frac{\alpha_{1}}{n}+\frac{\alpha_{2}}{n^{2}}+O\left(\frac{1}{n^{3}}\right) \quad\mbox{ and }\quad b_{n}=\beta_{0}+\frac{\beta_{1}}{n}+\frac{\beta_{2}}{n^{2}}+O\left(\frac{1}{n^{3}}\right),
\]
for $n\to\infty$. Further, three more conditions are to be assumed:
\begin{enumerate}
\item[a)] The quadratic polynomial $x^{2}+\alpha_{0}x+\beta_{0}$ has a double root $x_{0}$.
\item[b)] $2\beta_{1}=\alpha_{0}\alpha_{1}$.
\item[c)] The quadratic polynomial $y(y-1)x_{0}^{2}+(\alpha_{1}y+\alpha_{2})x_{0}+\beta_{2}$ has a double root $y_{0}$.
\end{enumerate}
Then equation~\eqref{eq:diff_eq_gen_wong_li} has two linearly independent solutions $\phi$ and $\xi$ of asymptotic behavior:
\begin{equation}
 \phi_{n}=x_{0}^{n}n^{y_{0}}\left[1+O\left(\frac{1}{n}\right)\right]
 \quad\mbox{ and }\quad
 \xi_{n}=\log(n)\,x_{0}^{n}n^{y_{0}}\left[1+O\left(\frac{1}{n}\right)\right],
\label{eq:sol_asympt_gen_wong_li}
\end{equation}
as $n\to\infty$.

If we write equation~\eqref{eq:diff_eq_z=0} into the form of~\eqref{eq:diff_eq_gen_wong_li}, we get the coefficients
\[
a_{n}=\frac{1-8(n+1+\nu)^{2}}{4(n+1+\nu)(n+2+\nu)}=-2+\frac{2}{n}-\frac{15+8\nu}{4n^{2}}+O\left(\frac{1}{n^{3}}\right), \quad n\to\infty,
\]
and
\[
b_{n}=\frac{n+\nu}{n+2+\nu}=1-\frac{2}{n}+\frac{4+2\nu}{n^{2}}+O\left(\frac{1}{n^{3}}\right) \quad n\to\infty.
\]
Further, one readily verifies that conditions a), b), c) are satisfied with $x_{0}=1$ and $y_{0}=-1/2$. Consequently, asymptotic formulas from~\eqref{eq:sol_asympt_gen_wong_li} become~\eqref{eq:sol_asympt_z=0}.

\subsection{Higher order terms in asymptotic expansion~\eqref{eq:asympt_x_0}}\label{subsec:a.2}

We complete the proof of Proposition~\ref{prop:asympt_x_0} by computing the two coefficients in expansion~\eqref{eq:asympt_x_0} by $\nu^{3}$ and $\nu^{4}$. To this end, we need to expand function~\eqref{eq:def_h_inproof} in powers of $\nu$ up to order four. This means to compute the first two terms in the expansion of function~\eqref{eq:def_U} and the leading term in~\eqref{eq:def_V}. A lengthy but straightforward computation yields
\[
 U(y;\nu)=U(y;0)+\partial_{\nu}U(y;0)\,\nu+O(\nu^{2}), \quad \nu\to0,
\]
where
\begin{align*}
 U(y;0)&=\frac{1}{(y+1)^{2}}\sum_{k=0}^{\infty}\frac{k!}{(k+2)(k+1)(y+2)_{k}},\\
 \partial_{\nu}U(y;0)&=\frac{2}{y+1}\sum_{k=0}^{\infty}\frac{k!}{(k+2)(k+1)(y+2)_{k}}\!\left[-\frac{y+2}{(y+1)^{2}}+\sum_{j=1}^{k}\frac{1}{j(y+1+j)}\right]
\end{align*}
and similarly 
\[
V(y;\nu)=V(y;0)+O(\nu), \quad \nu\to0,
\]
where
\[
 V(y;0)=-\sum_{k=0}^{\infty}\frac{k!}{(k+1)(k+y+1)^{2}(y+1)_{k}}.
\]
Moreover, one verifies that both remainders in $O(\nu^{2})$ and $O(\nu)$ are uniform in $y$ from a~neighborhood of zero.

Plugging the zero
\[
 y_{0}(\nu)=-\nu-\nu^{2}+y_{3}\nu^{3}+O(\nu^{4})
\]
into function
\[
 h(y;\nu)=y(y+1)+\nu+y\left[1+(y+1)U(y;0)\right]\nu^{2}-y\left[V(y;0)-(y+1)\partial_{\nu}U(y;0)\right]\nu^{3}+y\,O(\nu^{4}),
\]
dividing both sides of equation $h(y;\nu)=0$ by $\nu^{3}$, and sending $\nu\to0$ yields the equality
\[
 1-U(0;0)+y_{3}=0.
\]
Thus, coefficient $y_{3}$ can be expressed in the form
\[
 y_{3}=-1+U(0,0)=-1+\sum_{k=1}^{\infty}\frac{1}{(k+1)k^{2}}=-2+\frac{\pi^{2}}{6}.
\]

Repeating the same, this time with
\[
  y_{0}(\nu)=-\nu-\nu^{2}+(-1+U(0;0))\nu^{3}+y_{4}\nu^{4}+O(\nu^{5}),
\]
the coefficient at $h(y_{0}(\nu);\nu)$ by $\nu^{4}$ provides us with the equation for the coefficient $y_{4}$, which yields
\[
 y_{4}=-2+2U(0;0)+\partial_{\nu}U(0;0)-\partial_{y}U(0;0)-V(0;0).
\]
Unknown coefficients from the right-hand side can be expressed, for example, as follows:
\[
 \partial_{\nu}U(0;0)=\sum_{k=1}^{\infty}\frac{2}{(k+1)k^{2}}\left[-2+\sum_{j=2}^{k}\frac{1}{j(j-1)}\right]=-2\sum_{k=1}^{\infty}\frac{1}{k^{3}}=-2\zeta(3),
\]
\[
 \partial_{y}U(0;0)=-\sum_{k=1}^{\infty}\frac{1}{(k+1)k^{2}}\left[1+\sum_{j=1}^{k}\frac{1}{j}\right]=1-2\zeta(3)
 \]
and
\[
V(0;0)=-\sum_{k=1}^{\infty}\frac{1}{k^{3}}=-\zeta(3).
\]
Altogether, we obtain
\[
 y_{4}=-5+\frac{\pi^{2}}{3}+\zeta(3),
\]
which yields the desired expansion
\[
 y_{0}(\nu)=-\nu-\nu^{2}-\left(2-\frac{\pi^{2}}{6}\right)\nu^{3}-\left(5-\frac{\pi^{2}}{3}-\zeta(3)\right)\nu^{4}+O(\nu^{5}),
\]
for $\nu\to0$.

\bibliographystyle{acm}

\begin{thebibliography}{10}

\bibitem{abr-bus-car_ijmms03}
{\sc Abreu, L.~D., Bustoz, J., and Cardoso, J.~L.}
\newblock The roots of the third {J}ackson {$q$}-{B}essel function.
\newblock {\em Int. J. Math. Math. Sci.}, 67 (2003), 4241--4248.

\bibitem{akh_65}
{\sc Akhiezer, N.~I.}
\newblock {\em The classical moment problem and some related questions in
  analysis}.
\newblock Hafner Publishing Co., New York, 1965.
\newblock Translated by N. Kemmer.

\bibitem{ann-mas_mpcps09}
{\sc Annaby, M.~H., and Mansour, Z.~S.}
\newblock On the zeros of the second and third {J}ackson {$q$}-{B}essel
  functions and their associated {$q$}-{H}ankel transforms.
\newblock {\em Math. Proc. Cambridge Philos. Soc. 147}, 1 (2009), 47--67.

\bibitem{bec_jcam00}
{\sc Beckermann, B.}
\newblock Complex {J}acobi matrices.
\newblock vol.~127. 2001, pp.~17--65.
\newblock Numerical analysis 2000, Vol. V, Quadrature and orthogonal
  polynomials.

\bibitem{bec-smi_ms04}
{\sc Beckermann, B., and Castro~Smirnova, M.}
\newblock On the determinacy of complex {J}acobi matrices.
\newblock {\em Math. Scand. 95}, 2 (2004), 285--298.

\bibitem{bou-mas_oam21b}
{\sc Bouthat, L., and Mashreghi, J.}
\newblock {$L$}-matrices with lacunary coefficients.
\newblock {\em Oper. Matrices 15}, 3 (2021), 1045--1053.

\bibitem{bou-mas_oam21a}
{\sc Bouthat, L., and Mashreghi, J.}
\newblock The norm of an infinite {$L$}-matrix.
\newblock {\em Oper. Matrices 15}, 1 (2021), 47--58.

\bibitem{bou-mas_laa22}
{\sc Bouthat, L., and Mashreghi, J.}
\newblock The critical point and the $p$-norm of the {H}ilbert {$L$}-matrix.
\newblock {\em Linear Algebra Appl. 634\/} (2022), 1--14.

\bibitem{cho_amm83}
{\sc Choi, M.~D.}
\newblock Tricks or treats with the {H}ilbert matrix.
\newblock {\em Amer. Math. Monthly 90}, 5 (1983), 301--312.

\bibitem{dlmf}
{\it NIST Digital Library of Mathematical Functions}.
\newblock http://dlmf.nist.gov/, Release 1.0.27 of 2020-06-15.
\newblock F.~W.~J. Olver, A.~B. {Olde Daalhuis}, D.~W. Lozier, B.~I. Schneider,
  R.~F. Boisvert, C.~W. Clark, B.~R. Miller, B.~V. Saunders, H.~S. Cohl, and
  M.~A. McClain, eds.

\bibitem{gas-rah_04}
{\sc Gasper, G., and Rahman, M.}
\newblock {\em Basic hypergeometric series}, second~ed., vol.~96 of {\em
  Encyclopedia of Mathematics and its Applications}.
\newblock Cambridge University Press, Cambridge, 2004.
\newblock With a foreword by Richard Askey.

\bibitem{ism-let-val_siamjma89}
{\sc Ismail, M. E.~H., Letessier, J., and Valent, G.}
\newblock Quadratic birth and death processes and associated continuous dual
  {H}ahn polynomials.
\newblock {\em SIAM J. Math. Anal. 20}, 3 (1989), 727--737.

\bibitem{ism-mul_siamjma87}
{\sc Ismail, M. E.~H., and Mulla, F.~S.}
\newblock On the generalized {C}hebyshev polynomials.
\newblock {\em SIAM J. Math. Anal. 18}, 1 (1987), 243--258.

\bibitem{kal-sto_lma16}
{\sc Kalvoda, T., and \v{S}\v{t}ov\'{\i}\v{c}ek, P.}
\newblock A family of explicitly diagonalizable weighted {H}ankel matrices
  generalizing the {H}ilbert matrix.
\newblock {\em Linear Multilinear Algebra 64}, 5 (2016), 870--884.

\bibitem{kat_66}
{\sc Kato, T.}
\newblock {\em Perturbation theory for linear operators}.
\newblock Die Grundlehren der mathematischen Wissenschaften, Band 132.
  Springer-Verlag New York, Inc., New York, 1966.

\bibitem{koe-les-swa_10}
{\sc Koekoek, R., Lesky, P.~A., and Swarttouw, R.~F.}
\newblock {\em Hypergeometric orthogonal polynomials and their
  {$q$}-analogues}.
\newblock Springer Monographs in Mathematics. Springer-Verlag, Berlin, 2010.
\newblock With a foreword by Tom H. Koornwinder.

\bibitem{koe-swa_jmaa94}
{\sc Koelink, H.~T., and Swarttouw, R.~F.}
\newblock On the zeros of the {H}ahn-{E}xton {$q$}-{B}essel function and
  associated {$q$}-{L}ommel polynomials.
\newblock {\em J. Math. Anal. Appl. 186}, 3 (1994), 690--710.

\bibitem{koe-vanass_ca95}
{\sc Koelink, H.~T., and Van~Assche, W.}
\newblock Orthogonal polynomials and {L}aurent polynomials related to the
  {H}ahn-{E}xton {$q$}-{B}essel function.
\newblock {\em Constr. Approx. 11}, 4 (1995), 477--512.

\bibitem{lap-kre-sta_prep21}
{\sc Krej{\v c}i{\v r}{\' i}k, D., Laptev, A., and {\v S}tampach, F.}
\newblock Spectral enclosures and stability for non-self-adjoint discrete
  {S}chr\"{o}edinger operators on the half-line.
\newblock {\em Submitted}, arXiv:2111.08265 (2021).

\bibitem{mas_09}
{\sc Mashreghi, J.}
\newblock {\em Representation theorems in {H}ardy spaces}, vol.~74 of {\em
  London Mathematical Society Student Texts}.
\newblock Cambridge University Press, Cambridge, 2009.

\bibitem{mas-ran_amp19}
{\sc Mashreghi, J., and Ransford, T.}
\newblock Linear polynomial approximation schemes in {B}anach holomorphic
  function spaces.
\newblock {\em Anal. Math. Phys. 9}, 2 (2019), 899--905.

\bibitem{prudnikov-etal_vol3}
{\sc Prudnikov, A.~P., Brychkov, Y.~A., and Marichev, O.~I.}
\newblock {\em Integrals and series. {V}ol. 3}.
\newblock Gordon and Breach Science Publishers, New York, 1990.
\newblock More special functions, Translated from the Russian by G. G. Gould.

\bibitem{rai_71}
{\sc Rainville, E.~D.}
\newblock {\em Special functions}, first~ed.
\newblock Chelsea Publishing Co., Bronx, N.Y., 1971.

\bibitem{reed-simon_IV}
{\sc Reed, M., and Simon, B.}
\newblock {\em Methods of modern mathematical physics. {IV}. {A}nalysis of
  operators}.
\newblock Academic Press [Harcourt Brace Jovanovich, Publishers], New
  York-London, 1978.

\bibitem{rha_blms89}
{\sc Rhaly, Jr., H.~C.}
\newblock Terraced matrices.
\newblock {\em Bull. London Math. Soc. 21}, 4 (1989), 399--406.

\bibitem{ros_pams58}
{\sc Rosenblum, M.}
\newblock On the {H}ilbert matrix. {II}.
\newblock {\em Proc. Amer. Math. Soc. 9\/} (1958), 581--585.

\bibitem{simon_79}
{\sc Simon, B.}
\newblock {\em Trace ideals and their applications}, vol.~35 of {\em London
  Mathematical Society Lecture Note Series}.
\newblock Cambridge University Press, Cambridge-New York, 1979.

\bibitem{swarttow-phd92}
{\sc Swarttouw, R.~F.}
\newblock {\em The {H}ahn-{E}xton q-{B}essel function}.
\newblock ProQuest LLC, Ann Arbor, MI, 1992.
\newblock Thesis (Dr.)--Technische Universiteit Delft (The Netherlands).

\bibitem{tes_00}
{\sc Teschl, G.}
\newblock {\em Jacobi operators and completely integrable nonlinear lattices},
  vol.~72 of {\em Mathematical Surveys and Monographs}.
\newblock American Mathematical Society, Providence, RI, 2000.

\bibitem{sta_ieot17}
{\sc \v{S}tampach, F.}
\newblock The characteristic function for complex doubly infinite {J}acobi
  matrices.
\newblock {\em Integral Equations Operator Theory 88}, 4 (2017), 501--534.

\bibitem{sta-sto_laa13}
{\sc \v{S}tampach, F., and \v{S}t'ov\'{\i}\v{c}ek, P.}
\newblock The characteristic function for {J}acobi matrices with applications.
\newblock {\em Linear Algebra Appl. 438}, 11 (2013), 4130--4155.

\bibitem{sta-sto_sm13}
{\sc \v{S}tampach, F., and \v{S}\v{t}ov\'{\i}\v{c}ek, P.}
\newblock The {H}ahn-{E}xton {$q$}-{B}essel function as the characteristic
  function of a {J}acobi matrix.
\newblock {\em Spec. Matrices 1\/} (2013), 131--147.

\bibitem{sta-sto_jat16}
{\sc \v{S}tampach, F., and \v{S}\v{t}ov\'{\i}\v{c}ek, P.}
\newblock The {N}evanlinna parametrization for {$q$}-{L}ommel polynomials in
  the indeterminate case.
\newblock {\em J. Approx. Theory 201\/} (2016), 48--72.

\bibitem{li-wong_jcam92}
{\sc Wong, R., and Li, H.}
\newblock Asymptotic expansions for second-order linear difference equations.
\newblock vol.~41. 1992, pp.~65--94.
\newblock Asymptotic methods in analysis and combinatorics.

\end{thebibliography}

\end{document}